\documentclass[12pt,a4paper]{article}
\usepackage{amsfonts,amsthm,amsmath,amssymb,amscd}
\usepackage{fancybox}
\usepackage[latin1]{inputenc}
\usepackage[usenames]{color}
\usepackage{graphicx}
\usepackage{enumerate}
\usepackage{float}
\usepackage{fullpage}
\usepackage{epsfig,amssymb} 
\usepackage{subcaption}

\newtheorem{theorem}{Theorem}[section]
\newtheorem{corollary}[theorem]{Corollary}
\newtheorem{lemma}[theorem]{Lemma}
\newtheorem{proposition}[theorem]{Proposition}

\newtheorem{rem}[theorem]{Remark}

\begin{document}
	
\title{\textbf{Renormalized Area of Catenoids in the Hyperbolic Space}}

\author{\'Alvaro P\'ampano and Aaron J. Tyrrell}
\date{\today}

\maketitle

\begin{figure}[H]
	\begin{center}
		\includegraphics[height=5.5cm]{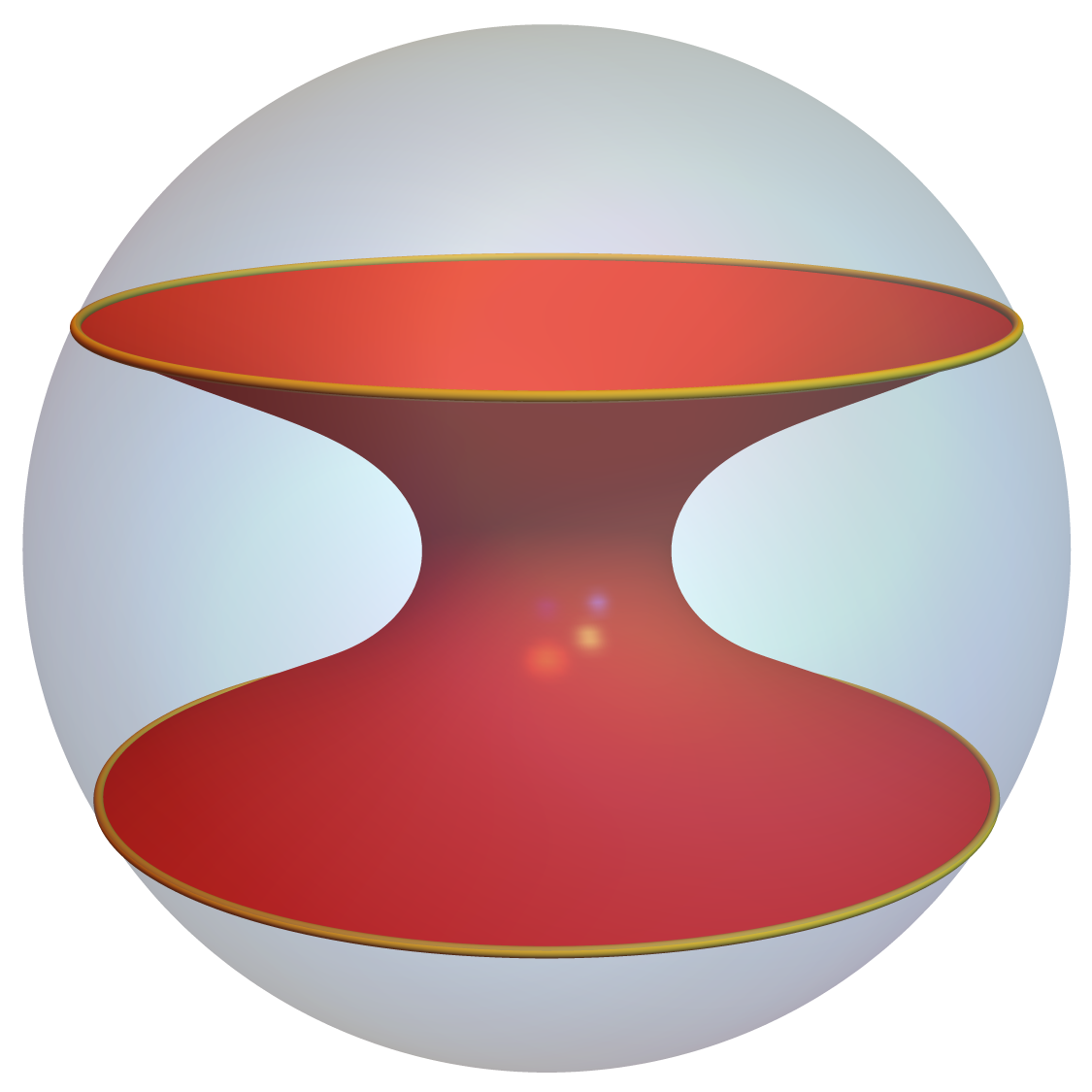}
	\end{center} 
\end{figure}

\begin{abstract} We show that the generating curves of non-totally geodesic spherical rotational minimal hypersurfaces (catenoids, for simplicity) of the hyperbolic spaces $\mathbb{H}^{2n+1}$ are $p$-elastic curves for $p=(2n-1)/(2n)$. We employ this variational characterization, together with the Chern--Gauss--Bonnet formulas for locally conformally flat manifolds, to present an explicit expression for the renormalized area of catenoids in terms of hyperelliptic integrals. Further analyzing these special integrals, we show that the renormalized area of catenoids varies continuously from negative infinity to twice the renormalized area of the totally geodesic hypersurfaces $\mathbb{H}^{2n}\subset\mathbb{H}^{2n+1}$. Therefore, we conclude that the renormalized area is not bounded below and that, when $n$ is even, the renormalized area of minimal hypersurfaces in $\mathbb{H}^{2n+1}$ does not have a sign.\\

\end{abstract}

\section{Introduction}

Minimal submanifolds in hyperbolic spaces with prescribed boundary at infinity have been a fruitful topic of investigation since the seminal works of Anderson \cite{A1,A2}. These submanifolds have infinite area, but under certain assumptions near their boundary they possess a Hadamard regularization known as the renormalized area $\mathcal{A}_R$. The renormalized area of minimal submanifolds first arose in physics as a central tool in the AdS/CFT correspondence (see, for instance, \cite{Physics} and references therein).

Introduced by Graham and Witten \cite{Graham,GW} in the mathematics literature, the renormalized area $\mathcal{A}_R$ is a global invariant of the submanifold, provided that the submanifold is even-dimensional. The renormalized area of minimal surfaces in $\mathbb{H}^3$ was thoroughly studied in \cite{AM}, where an explicit formula was presented. For the case of $4$-dimensional minimal hypersurfaces in $\mathbb{H}^5$ there is also an explicit Gauss-Bonnet type formula for the renormalized area \cite{T}. More recently, a general treatment of renormalized area of even-dimensional minimal hypersurfaces has been developed in \cite{GT,New}, while of just asymptotically minimal ones can be found in \cite{PT}. However, there seem to be very few works in which the renormalized area is computed for specific examples. For instance, in \cite{GT,CQY,E}, the renormalized area $\mathcal{A}_R$ of totally geodesic hypersurfaces of $\mathbb{H}^{2n+1}$ has been computed, obtaining
$$\mathcal{A}_R(\mathbb{H}^{2n})=\frac{(-4\pi)^n n!}{(2n)!}\,.$$
This formula is accessible through direct calculations of expansions of truncated parts of $\mathbb{H}^{2n}$. (Actually, in \cite{GT,CQY,E} it is the renormalized volume that is computed. Both quantities coincide as shown in Section 7.1 of \cite{MK}.)

The main goal of this paper is to compute the renormalized area $\mathcal{A}_R$ of catenoids in $\mathbb{H}^{2n+1}$. Catenoids in $\mathbb{H}^{2n+1}$ are a one-parameter family of non-totally geodesic spherical rotational minimal hypersurfaces \cite{DCD}. Hence, our results will give a one variable function for $\mathcal{A}_R$ rather than just a specific value, as is the case of $\mathbb{H}^{2n}$. It will then be possible to extract some information regarding the renormalized area of minimal hypersurfaces, including intervals contained in the range of $\mathcal{A}_R$.

Consider the family of catenoids $\{M_c\}$, $c\in(0,\infty)$, in $\mathbb{H}^{2n+1}$ parameterized in terms of the parameter $c>0$ which essentially describes the radius of the neck of the catenoids (for details, see Part (iv) of Proposition \ref{prop} and the footnote on Page 9; see also the definition of the warping function in Theorem \ref{t1}). In Theorem \ref{t4}, we show that the renormalized area $\mathcal{A}_R$ of a catenoid $M_c$, $c\in(0,\infty)$, in $\mathbb{H}^{2n+1}$ is given by the hyperelliptic integral
$$\mathcal{A}_R(M_c)=\frac{2\pi^n}{(n-1)!c^{2n-1}}\int_0^\alpha \frac{q_n(u)}{\sqrt{u(-u^{2n}+c^2u+(2n-1)^2)}}\,du\,,$$
where $\alpha$ is the only positive root of $Q_{n,c}(u)=-u^{2n}+c^2u+(2n-1)^2$ and $q_n$ is the polynomial of degree $n(2n-1)$ given by
$$q_n(u)=\sum_{i=1}^n(-1)^i\begin{pmatrix} n \\ i \end{pmatrix}(2n-1)^{2(n-i)}(2i-1)u^{n(2i-1)}.$$
Although the integrand is singular at the endpoints, the hyperelliptic integral above is convergent for every $c\in(0,\infty)$ and positive integer $n$ (see Remark \ref{convergence}).

Further analyzing the asymptotic behavior of the hyperelliptic integral $\mathcal{A}_R(M_c)$ we prove in Theorem \ref{t5} that $\mathcal{A}_R(M_c)$ is a continuous function of the variable $c\in(0,\infty)$ satisfying
$$\lim_{c\to 0^+}\mathcal{A}_R(M_c)=-\infty\,,\quad\quad\quad \lim_{c\to \infty}\mathcal{A}_R(M_c)=2\mathcal{A}_R(\mathbb{H}^{2n})\,.$$
As a consequence, we deduce that the renormalized area of hypersurfaces in $\mathbb{H}^{2n+1}$ is not bounded below ($\mathcal{A}_R$ is not bounded above on the space of asymptotically minimal hypersurfaces either, since one can always deform a hypersurface fixing a neighborhood of its boundary at infinity and increasing its area arbitrarily). In addition, we also deduce that, when $n$ is even, the renormalized area of minimal hypersurfaces of $\mathbb{H}^{2n+1}$ does not have a sign.

The expression of $\mathcal{A}_R(M_c)$ given above is particularly enlightening when $n=1$, that is for catenoids $M_c$, $c\in(0,\infty)$, in $\mathbb{H}^3$. For this specific case, the integral appearing in $\mathcal{A}_R(M_c)$ is elliptic and can be rewritten as a combination of the standard complete elliptic integrals of the first ($K$) and second ($E$) kind, namely, (c.f., Theorem \ref{t2})
$$\mathcal{A}_R(M_c)=\frac{4\pi}{\sqrt{2\zeta^2-1}}\left((1-\zeta^2)K(\zeta)-E(\zeta)\right),$$
where
$$\zeta=\sqrt{\frac{c^2+\sqrt{c^2+4}}{2\sqrt{c^2+4}}\,}\in (\sqrt{2}/2,1).$$
Employing the well known identities of complete elliptic integrals (see, for instance, \cite{BF,GR}) we then show that the function $\mathcal{A}_R(M_c):c\in(0,\infty)\longmapsto \mathbb{R}$ increases continuously from $-\infty$ until it attains its maximum 
$$\mathcal{A}_R^+=-4\pi\sqrt{2\zeta_+^2-1}\,E(\zeta_+)\in(-4\pi,-\pi^2)\,,$$
where $\zeta_+\in(\sqrt{2}/2,1)$ is the unique number satisfying $K(\zeta_+)=2E(\zeta_+)$. The associated unique value $c_+\in(0,\infty)$ is approximately $1.3118$. The function $\mathcal{A}_R(M_c)$ then decreases approaching $-4\pi$ from above (see the graph of this function on Figure \ref{G2}). Moreover, we show in Corollary \ref{cor} that the maximum $\mathcal{A}_R^+$ is attained at the catenoid $M_{c_+}$ with maximum height $\Psi_1$ (defined in Proposition \ref{prop} as half of the Euclidean distance between the boundary circles in the ball model for $\mathbb{H}^3$). An illustration of the catenoid $M_{c_+}$ with maximum height and maximum renormalized area can be found in the front page. The geometric behavior of the function $\mathcal{A}_R(M_c)$ seems to be analogous in higher dimensions (see Figure \ref{G4} for the case $n=2$).

To achieve our main results regarding the expression of the renormalized area $\mathcal{A}_R$ of catenoids $M_c$, $c\in(0,\infty)$, in $\mathbb{H}^{2n+1}$, we will employ the Chern--Gauss--Bonnet formulas for locally conformally flat manifolds (see \cite{CC,CS,V} for details) and exploit a connection between rotational minimal hypersurfaces of space forms and the theory of (generalized) elastic curves (see \cite{AGP,AP} for the case of surfaces). 

The theory of generalized elastic curves originated with the works of the Bernoulli family and Euler. More precisely, D. Bernoulli suggested to Euler to analyze extrema of the functionals
$$\mathbf{\Theta}_p(\gamma)=\int_\gamma \kappa^p\,ds\,,$$
where $\kappa$ denotes the curvature of the curve $\gamma$ and $s$ is its arc length parameter. Nowadays, we refer to equilibria for $\mathbf{\Theta}_p$ as $p$-elastic curves and they are still a central field of research (see, for instance, \cite{AGP,AP} and references therein). Arguably, the case $p=2$, which is known as the classical bending energy, has received the most attention because of its physical meaning and because it appears in a vast number of applications (for a modern treatment, see \cite{LS}). Nonetheless, other values of $p$ have also appeared in the literature. For instance, the case $p=1/2$ which is closely related to the present paper, was first studied by Blaschke \cite{B}, who showed that equilibria in the Euclidean plane $\mathbb{R}^2$ are catenaries and, hence, the generating curves of rotational minimal surfaces in $\mathbb{R}^3$.

In \cite{AGP}, the generating curve of a spherical rotational surface in the hyperbolic space $\mathbb{H}^3$ was characterized as a $1/2$-elastic curve. In Theorem \ref{t1} below we extend this characterization to arbitrary dimensions, showing that the generating curves of catenoids in $\mathbb{H}^{2n+1}$ are $p$-elastic curves for $p=(2n-1)/(2n)\in[1/2,1)$. The non-constant curvatures of these curves satisfy a first-order ordinary differential equation, which allows us to make a change of variable in the hypersurface integrals arising in our Gauss--Bonnet type identities (see Proposition \ref{prop2}), concluding with the expression of $\mathcal{A}_R(M_c)$ given above.

\section{Area of Minimal Hypersurfaces}

Let $n$ be a fixed positive integer and denote by $\mathbb{H}^{2n+1}$ the $(2n+1)$-dimensional simply-connected hyperbolic space of constant sectional curvature $-1$. Throughout this paper, we will consider the hyperquadric model for $\mathbb{H}^{2n+1}$, as described below.

Let $(x_1,...,x_{2n+1},z)$ be the standard coordinates of $\mathbb{R}^{2n+2}$. Endow the space $\mathbb{R}^{2n+2}$ with the canonical index one metric
$$g=\sum_{i=1}^{2n+1}dx_i^2 - dz^2\,.$$
The pair $(\mathbb{R}^{2n+2},g)$ is the Lorentz-Minkowski $(2n+2)$-space, simply denoted by $\mathbb{L}^{2n+2}$. The hyperbolic space $\mathbb{H}^{2n+1}$ is the space-like hypersurface of $\mathbb{L}^{2n+2}$ given by the top part of the hyperquadric
$$\mathbb{H}^{2n+1}=\{(x_1,...,x_{2n+1},z)\in\mathbb{R}^{2n+2}\,\lvert\,\sum_{i=1}^{2n+1}x_i^2-z^2=1\,,z>0\}\,,$$
endowed with the Riemannian metric induced from $\mathbb{L}^{2n+2}$. For later use, the hyperbolic spaces of even-dimension $\mathbb{H}^{2n}$ will be considered to be $\mathbb{H}^{2n+1}\cap\{x_2=0\}$. In particular, the hyperbolic plane $\mathbb{H}^2$ will be $\mathbb{H}^3\cap\{x_2=0\}$.

For visualization purposes, when $n=1$, we will identify $\mathbb{H}^3$ with the unit ball $\mathbb{B}\subset\mathbb{R}^3$ centered at the origin and endowed with the Poincar\'e metric by means of the isometry
\begin{equation}\label{ident}
	(x_1,x_2,x_3,z)\in\mathbb{H}^3\subset\mathbb{R}^4\longmapsto \frac{1}{1+z}(x_1,x_2,x_3)\in\mathbb{B}\subset\mathbb{R}^3\,.
\end{equation}
This identification gives rise to the Poincar\'e ball model for $\mathbb{H}^3$. If $x_2=0$, the above isometry provides the Poincar\'e disc model for the hyperbolic plane $\mathbb{H}^2$.

Consider a smooth and compact (with boundary) hypersurface $M$ embedded in $\mathbb{H}^{2n+1}$. The area of the hypersurface $M$ in $\mathbb{H}^{2n+1}$ is
\begin{equation}\label{area}
	\mathcal{A}(M)=\int_M 1\,dA\,,
\end{equation}
where $dA$ is the volume element on $M$ induced from the metric of $\mathbb{H}^{2n+1}\subset\mathbb{L}^{2n+2}$. The scalar curvature of $M$ is defined as
\begin{equation}\label{lambda}
	\lambda=\frac{1}{2n(2n-1)}\sum_{i=1}^{2n}\sum_{j=1}^{2n} {\rm Rm}(e_i,e_j,e_i,e_j)=\frac{1}{2n(2n-1)}\sum_{i=1}^{2n} {\rm Ric}(e_i,e_i)\,,
\end{equation}
where ${\rm Rm}$ stands for the Riemann curvature of $M$, ${\rm Ric}$ is the Ricci tensor, and $\{e_1,...,e_{2n}\}$ is any orthonormal frame of the tangent bundle $TM$. Observe that, for convenience, here we are following \cite{PT} and have opted to include a coefficient depending on the dimension in our definition of the scalar curvature.

Let $\xi$ be a unit normal vector field to $M$ and denote by $B$ the second fundamental form of $M$. The eigenvalues of $B$ are the principal curvatures of $M$, which are represented by $\kappa_1,...,\kappa_{2n}$. The mean curvature $H$ of $M$ is
\begin{equation}\label{H}
	H=\frac{1}{2n}{\rm Trace}(B)=\frac{1}{2n}\sum_{i=1}^{2n}\kappa_i\,.
\end{equation}
A hypersurface $M$ is minimal if the mean curvature $H$ is identically zero.

Employing the Gauss equation, we can relate the scalar curvature $\lambda$ of a minimal hypersurface $M$ to the squared norm of the second fundamental form $\lvert B\rvert^2$. Indeed,
\begin{eqnarray*}
	\frac{1}{2n(2n-1)}\lvert B\rvert^2&=&\frac{1}{2n(2n-1)}\sum_{i=1}^{2n}\kappa_i^2=\frac{-1}{2n(2n-1)}\sum_{i\neq j} \kappa_i\kappa_j\\
	&=&\frac{1}{2n(2n-1)}\sum_{i\neq j}\left(\widetilde{\rm Rm}(e_i,e_j,e_i,e_j)-{\rm Rm}(e_i,e_j,e_i,e_j)\right)\\
	&=&\frac{1}{2n(2n-1)}\sum_{i\neq j}(-1)-\lambda=-1-\lambda\,,
\end{eqnarray*}
where $\widetilde{\rm Rm}$ is the Riemann curvature of $\mathbb{H}^{2n+1}$. The second equality is a consequence of the minimality of $M$. In fact, since $M$ is minimal, we deduce from \eqref{H} that $0=(\kappa_1+\cdots+\kappa_{2n})^2$ holds and, expanding the right-hand side, we obtain the desired equality. The third equality is the Gauss equation applied to the frame of principal directions $\{e_1,...,e_{2n}\}$.

In conclusion, for a minimal hypersurface $M$ of $\mathbb{H}^{2n+1}$, we have the following relation
\begin{equation}\label{lambdaB}
	\lambda=-1-\frac{1}{2n(2n-1)}\lvert B\rvert^2\,,
\end{equation}
between its scalar curvature and the squared norm of its second fundamental form $\lvert B\rvert^2$.

From this identity, we will next deduce a formula for the area of $M$.

\begin{proposition}\label{prop1} Let $M$ be a compact minimal hypersurface embedded in $\mathbb{H}^{2n+1}$. The area of $M$ satisfies
	\begin{equation}\label{formula}
		\mathcal{A}(M)=(-1)^n\int_M \lambda^n\,dA-\sum_{i=1}^n \frac{1}{\left(2n(2n-1)\right)^i}\begin{pmatrix} n \\ i \end{pmatrix}\int_M \lvert B\rvert^{2i}\,dA\,,
	\end{equation}
where $\lambda$ is the scalar curvature of $M$ and $\lvert B\rvert^2$ is the squared norm of the second fundamental form of $M$.
\end{proposition}
\textit{Proof.} Let $M$ be a minimal hypersurface in $\mathbb{H}^{2n+1}$. From the relation \eqref{lambdaB} between the scalar curvature and the squared norm of the second fundamental form, we get
\begin{eqnarray*}
	\lambda^n&=&\left(-1-\frac{\lvert B\rvert^2}{2n(2n-1)}\right)^n=(-1)^n\left(1+\frac{\lvert B\rvert^2}{2n(2n-1)}\right)^n=(-1)^n\sum_{i=0}^{n}\begin{pmatrix} n \\ i \end{pmatrix} \frac{\lvert B\rvert^{2i}}{\left(2n(2n-1)\right)^i}\\
	&=&(-1)^n+(-1)^n\sum_{i=1}^n \begin{pmatrix} n \\ i \end{pmatrix} \frac{\lvert B\rvert^{2i}}{\left(2n(2n-1)\right)^i}\,.
\end{eqnarray*}
Multiplying by $(-1)^n$ on both sides, integrating over $M$, and rearranging we prove the statement. \hfill$\square$
\\

The above formula will be essential in the computation of the renormalized area. We will next further develop it for the kind of hypersurfaces under consideration, namely, rotational minimal hypersurfaces of $\mathbb{H}^{2n+1}$.

A hypersurface $M$ in $\mathbb{H}^{2n+1}$ is a (spherical) rotational hypersurface if it stays invariant under the action of the orthogonal group $O(2n)$, considered as a subgroup of isometries of the ambient space $\mathbb{H}^{2n+1}$. The orbit of a point in $\mathbb{H}^{2n+1}$ under the action of $O(2n)$ is a $(2n-1)$-dimensional sphere. The hypersurface $M$ can be described as the evolution of an arc length parameterized curve $\gamma$ in $\mathbb{H}^2=\mathbb{H}^3\cap\{x_2=0\}$ under the action of $O(2n)$. We refer to the curve $\gamma(s)=(x_1(s),0,x_3(s),z(s))$ as the generating curve (or, profile curve) of $M$.

From Proposition 3.2 of \cite{DCD}, we have that a rotational hypersurface has, at most, two distinct principal curvatures which, up to sign depending on the choice of orientation, are
\begin{equation}\label{pcurvatures}
	\kappa_1=\frac{x_1''-x_1}{\sqrt{1+x_1^2-(x_1')^2}}\,,\quad\quad\quad \kappa_2=\cdots=\kappa_{2n}=-\frac{\sqrt{1+x_1^2-(x_1')^2}}{x_1}\,.
\end{equation}
Here, we are denoting by $\left(\,\right)'$ the derivative with respect to the arc length parameter $s$ of the generating curve $\gamma$. The principal curvature $\kappa_1$ corresponds to the curvature $\kappa$ of $\gamma$ (up to re-orientation of this curve). 

It follows from \eqref{H} and \eqref{pcurvatures} that if a rotational hypersurface $M$ of $\mathbb{H}^{2n+1}$ is minimal, then
$$0=\kappa_1+\kappa_2+\cdots+\kappa_{2n}=\kappa_1+(2n-1)\kappa_2=\kappa+(2n-1)\kappa_2\,,$$
and so 
\begin{equation}\label{pk}
	\kappa_1=\kappa\,,\quad\quad\quad \kappa_i=-\frac{\kappa}{2n-1}\,,
\end{equation} 
for every $i=2,...,2n$.

We then specialize Proposition \ref{prop1} for rotational minimal hypersurfaces.

\begin{proposition}\label{prop2} Let $M$ be a compact part of a rotational minimal hypersurface in $\mathbb{H}^{2n+1}$. Then the area of $M$ is given by
	\begin{equation}\label{formulak}
		\mathcal{A}(M)=(-1)^n\int_M \lambda^n\,dA-\sum_{i=1}^{n}\frac{1}{(2n-1)^{2i}}\begin{pmatrix} n \\ i \end{pmatrix}\int_M \kappa^{2i}\,dA\,,
	\end{equation}
where $\lambda$ is the scalar curvature of $M$ and $\kappa$ is the curvature of the generating curve of $M$.
\end{proposition}
\textit{Proof.} Let $M$ be a part of a rotational minimal hypersurface in $\mathbb{H}^{2n+1}$. From $\kappa_1=\kappa$ and $\kappa_2=\cdots=\kappa_{2n}=-\kappa/(2n-1)$ (see \eqref{pk}), we compute
\begin{equation}\label{B2}
	\lvert B\rvert^2=\sum_{i=1}^{2n}\kappa_i^2=\kappa^2+\sum_{i=2}^{2n}\frac{\kappa^2}{(2n-1)^2}=\frac{2n}{2n-1}\,\kappa^2\,.
\end{equation}
We substitute this on \eqref{formula} and simplify to prove the result. \hfill$\square$
\\

Rotational minimal hypersurfaces $M$ in $\mathbb{H}^{2n+1}$ approach the ideal boundary $\partial_\infty\mathbb{H}^{2n+1}$, and so their areas are infinite. Nonetheless, the area functional possesses a well defined Hadamard regularization, typically referred to as the renormalized area.

The renormalized area $\mathcal{A}_R$ of a hypersurface $M$ is the constant term in the expansion of the area of the truncated hypersurfaces $M_\epsilon=M\cap\{z\leq 1/\epsilon\}$, for $\epsilon>0$ sufficiently small (recall that we are using the hyperquadric model for $\mathbb{H}^{2n+1}\subset\mathbb{L}^{2n+2}$). The truncated hypersurfaces $M_\epsilon$ are compact and embedded in $\mathbb{H}^{2n+1}$. Hence, we can use \eqref{formulak} to understand their area.

If the curvature $\kappa$ of the generating curve of a rotational minimal hypersurface $M$ vanishes identically, $M$ is totally geodesic, that is, an even-dimensional hyperbolic space $\mathbb{H}^{2n}$. The renormalized area of $\mathbb{H}^{2n}$ is well known \cite{CQY,E} (see also Appendix A below). On the contrary, if the curvature $\kappa$ of the generating curve of the rotational minimal hypersurface is not identically zero, we need to understand the integrals of its powers on the right-hand side of \eqref{formulak}. For simplicity, we refer to non-totally geodesic (spherical) rotational minimal hypersurfaces $M$ of $\mathbb{H}^{2n+1}$ as catenoids. 

\section{Variational Characterization of Catenoids}

In this section we will describe a variational characterization of the generating curves of catenoids in $\mathbb{H}^{2n+1}$. We will see that these curves are $p$-elastic curves for $p=(2n-1)/(2n)$. Therefore, we begin by briefly recalling the theory of $p$-elastic curves in the hyperbolic plane $\mathbb{H}^2$.

Consider the $p$-energy functional 
\begin{equation}\label{energy}
	\mathbf{\Theta}_p(\gamma)=\int_\gamma \kappa^p\,ds\,,
\end{equation}
for $p=(2n-1)/(2n)$ acting on the space of convex curves immersed in $\mathbb{H}^2$. Regardless of the boundary conditions, a curve $\gamma$ with non-constant curvature $\kappa$ which is critical for $\mathbf{\Theta}_p$ with $p=(2n-1)/(2n)$ must satisfy the first-order ordinary differential equation
\begin{equation}\label{ODE}
	(u')^2=\frac{4u^2}{(2n-1)^2}\,Q_{n,a}(u)\,,
\end{equation}
where $u=\kappa^{1/n}$ is the normalized curvature and $Q_{n,a}$ is the polynomial
\begin{equation}\label{Q}
	Q_{n,a}(u)=-u^{2n}+4n^2 a u+(2n-1)^2\,. 
\end{equation}
The constant $a\in\mathbb{R}$ in \eqref{Q} represents the magnitude of the momentum, which is a certain vector field in $\mathbb{L}^3$. For our purposes, it will be sufficient to consider it to be positive. Hence, we introduce the constant $c=2n\sqrt{a}>0$ as our parameter for the family of critical curves with $a>0$, and denote $Q_{n,c}=Q_{n,a}$. In this case, a simple analysis of the polynomial $Q_{n,c}$ given in \eqref{Q} shows that it only possesses one positive root $\alpha$. For convenience, we take the arc length parameter $s$ of the critical curves to satisfy the initial condition $u(0)=\alpha$ for \eqref{ODE}.

We refer to equation \eqref{ODE} as the Euler--Lagrange equation\footnote{Actually, equation \eqref{ODE} is the first integral of the \emph{proper} Euler--Lagrange equation associated to $\mathbf{\Theta}_p$. When the curvature $\kappa$ (and, hence, $u$) is not constant, both equations are equivalent.} associated to $\mathbf{\Theta}_p$. Curves $\gamma$ whose non-constant curvature $\kappa$ satisfies the Euler--Lagrange equation \eqref{ODE} for $a>0$ are called $n$-catenaries, in analogy with their Euclidean counterpart. The functional $\mathbf{\Theta}_{1/2}$ ($n=1$) acting on curves in the plane $\mathbb{R}^2$ was studied by Blaschke \cite{B}, who proved that critical curves are catenaries. In \cite{AGP}, the functional $\mathbf{\Theta}_{1/2}$ was extended to act on curves immersed in Riemannian and Lorentzian space forms and it was shown that critical curves are the generating curves of rotational surfaces with zero mean curvature (see also \cite{MP} for an analysis of the functional $\mathbf{\Theta}_{1/2}$ acting on the space of curves in $\mathbb{H}^2$ with fixed length).

A $n$-catenary associated to the parameter $c=2n\sqrt{a}>0$ in the hyperbolic plane $\mathbb{H}^2$ can be parameterized in terms of its arc length parameter, up to rigid motions of $\mathbb{H}^2$, as
\begin{equation}\label{param}
	\gamma_c(s)=\frac{1}{c\sqrt{u}}\left(2n-1,0,\sqrt{c^2 u+(2n-1)^2}\sinh v(s),\sqrt{c^2 u+(2n-1)^2}\cosh v(s)\right),
\end{equation}
where $u\equiv u(s)=\kappa^{1/n}(s)$ is the normalized curvature and
$$v(s)=-c\int_0^s \frac{u^{n+1/2}}{c^2 u+(2n-1)^2}\,d\sigma\,,$$
represents the (hyperbolic) angular progression. The parameterization \eqref{param} can be obtained adapting the computations of Langer and Singer \cite{LS} involving Killing vector fields along curves to the present case (a general computation can be found in \cite{AP}; see also the proof of Theorem \ref{t1} below).

From the Euler--Lagrange equation \eqref{ODE} and the parameterization of $n$-catenaries given above, we can deduce some geometric properties of these curves.

\begin{proposition}\label{prop} Let $\widetilde{\gamma}_c\subset\mathbb{B}$ be the identification of the $n$-catenary $\gamma_c$ for $c>0$ as a curve immersed in the Poincar\'e disc model for the hyperbolic plane $\mathbb{H}^2$. Then, up to rigid motions of $\mathbb{H}^2$:
	\begin{enumerate}[(i)]
		\item The curvature $\kappa_c$ of $\widetilde{\gamma}_c$ increases monotonically from $0$, when $s\to -\infty$, to its maximum $\alpha^n$, at $s=0$, and then decreases monotonically to $0$, as $s\to \infty$.
		\item When $\kappa_c\to 0^+$, the trajectory of $\widetilde{\gamma}_c$ approaches the ideal boundary of $\mathbb{H}^2$ at a right angle. At the maximum curvature $\alpha^n$, the trajectory of $\widetilde{\gamma}_c$ meets $\{x_3=0\}$	orthogonally.
		\item The curve $\widetilde{\gamma}_c$ is symmetric with respect to $\{x_3=0\}$.
		\item The minimum (Euclidean) distance from $\widetilde{\gamma}_c$ to the center of the Poincar\'e disc is given by
		$$\Phi_n(c)=\frac{2n-1}{c\sqrt{\alpha}+\sqrt{c^2\alpha+(2n-1)^2}}=\frac{2n-1}{c\sqrt{\alpha}+\alpha^n}=\frac{1}{2n-1}\left(\alpha^n-c\sqrt{\alpha}\right).$$
		\item The maximum (Euclidean) distance from $\widetilde{\gamma}_c$ to $\{x_3=0\}$ (for simplicity, the height) is given by
		$$\Psi_n(c)=\tanh\left(\frac{(2n-1)c}{2}\int_0^\alpha \frac{u^n}{(c^2 u+(2n-1)^2)\sqrt{u\, Q_{n,c}(u)}}\,du\right).$$
	\end{enumerate}
\end{proposition}
\textit{Proof.} Consider a $n$-catenary $\gamma_c$ for $c=2n\sqrt{a}>0$. By definition, the non-constant curvature $\kappa_c$ of $\gamma_c$ must satisfy the differential equation \eqref{ODE}, where $u=\kappa_c^{1/n}$ (recall that $\gamma_c$ is convex and so $\kappa_c$ is positive and so is $u$). If the curvature $\kappa_c$ attains its maximum or minimum, then $u'=0$ must hold. From \eqref{ODE}, this happens when either $u=0$ (understood as a limiting case) or when the polynomial $Q_{n,c}(u)=0$ at some positive value of $u$, hence $u=\alpha$. This shows property (i). 

The remaining properties follow immediately from the parameterization \eqref{param} and the identification \eqref{ident}, when necessary. Hence, it suffices to show that \eqref{param} is an arc length parameterization of the $n$-catenary $\gamma_c$ for $c>0$. First, observe that from \eqref{param}, $g(\gamma_c(s),\gamma_c(s))=-1$ and so the curve is contained in $\mathbb{H}^2$. Further, differentiating \eqref{param} with respect to $s$ and using the definition of the hyperbolic angular progression $v(s)$, we obtain $g(\gamma_c'(s),\gamma_c'(s))=1$, which verifies that \eqref{param} is a parameterization by arc length. Finally, computing the second derivative and employing the Gauss formula to relate the curvature of $\gamma_c$ as a curve in $\mathbb{L}^4$ and as a curve in $\mathbb{H}^2\subset\mathbb{H}^3$, we deduce that the normalized curvature $u_c=\kappa_c^{1/n}$ of \eqref{param} is precisely a solution of \eqref{ODE}. Therefore, up to rigid motions, the curve given in \eqref{param} is the $n$-catenary $\gamma_c$. \hfill$\square$
\\

\begin{figure}[h!]
	\centering
	\includegraphics[height=5.58cm]{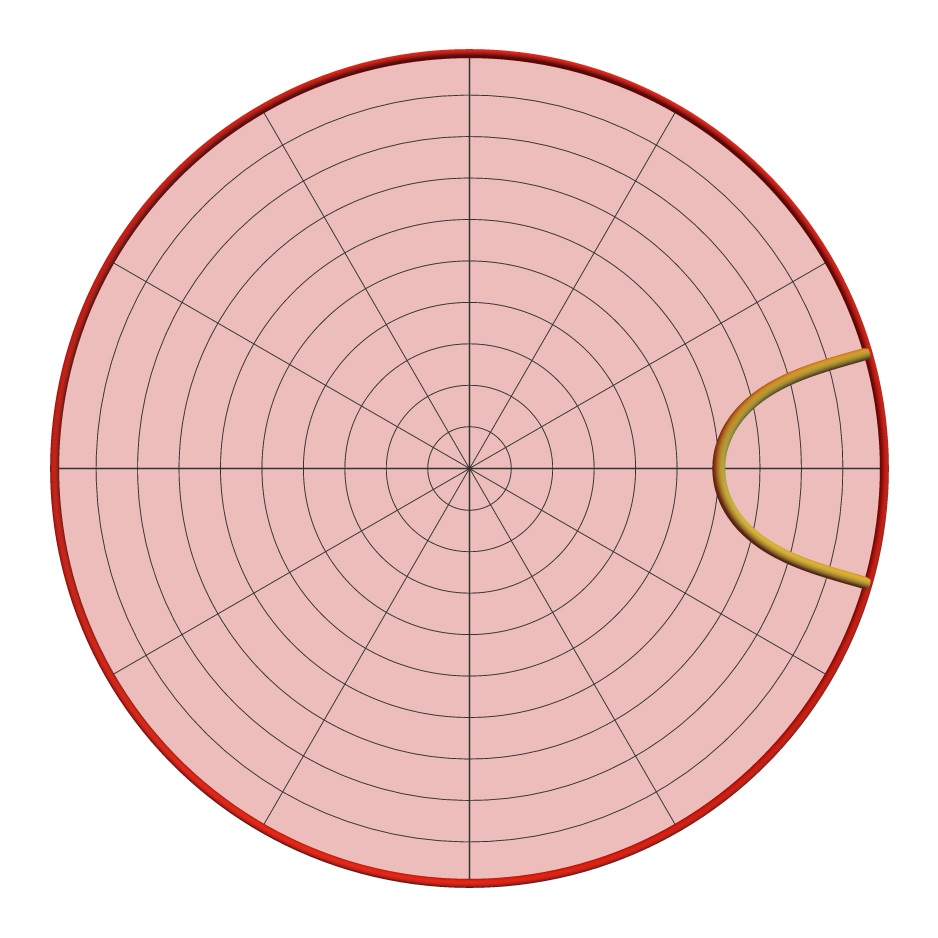}
	\includegraphics[height=5.58cm]{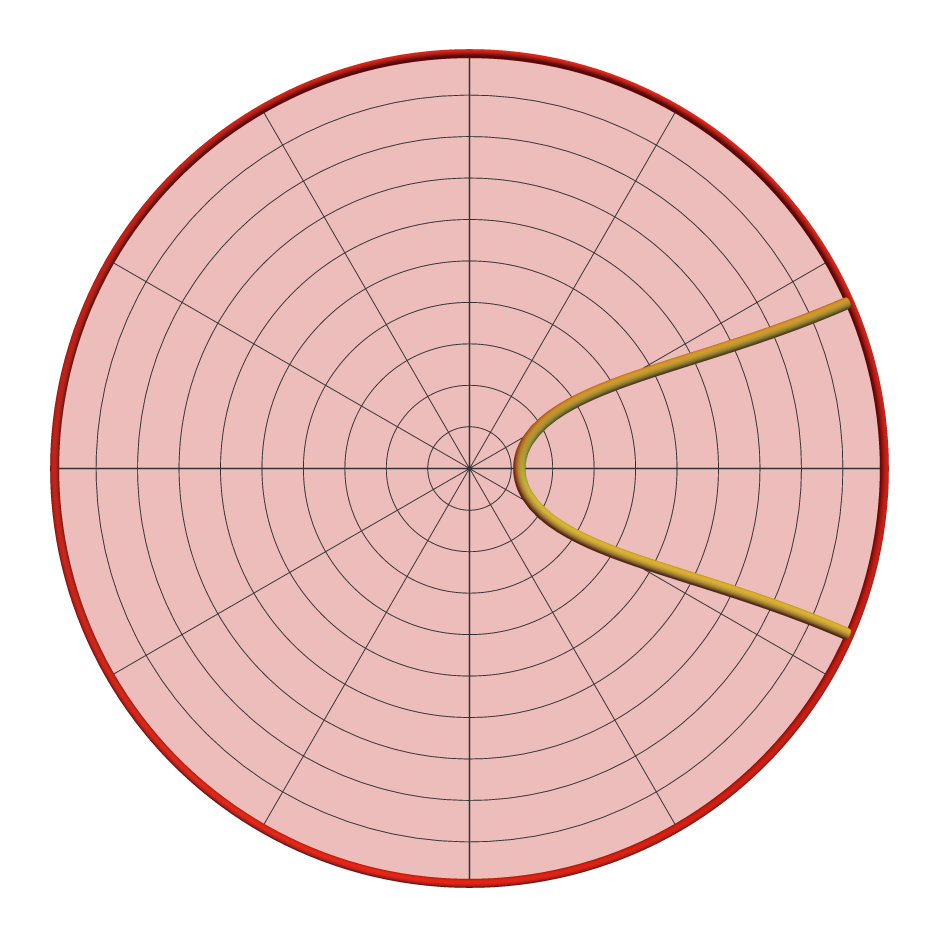}
	\includegraphics[height=5.58cm]{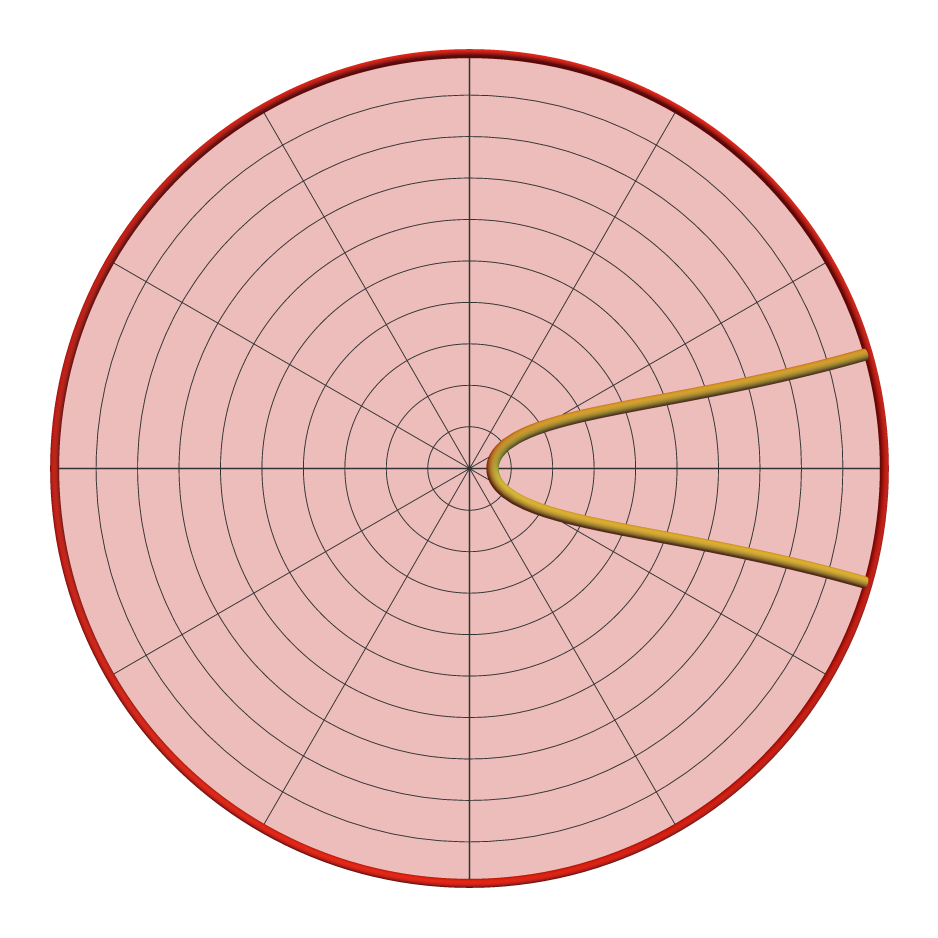}
	\caption{\small{Three $n$-catenaries for $n=1$ in the hyperbolic plane $\mathbb{H}^2$ corresponding to the values $c=0.5$, $c=2$, and $c=3$, respectively. They are represented in the Poincar\'e disc model.}}
	\label{F1}
\end{figure}

\begin{figure}[h!]
	\centering
	\includegraphics[height=5.58cm]{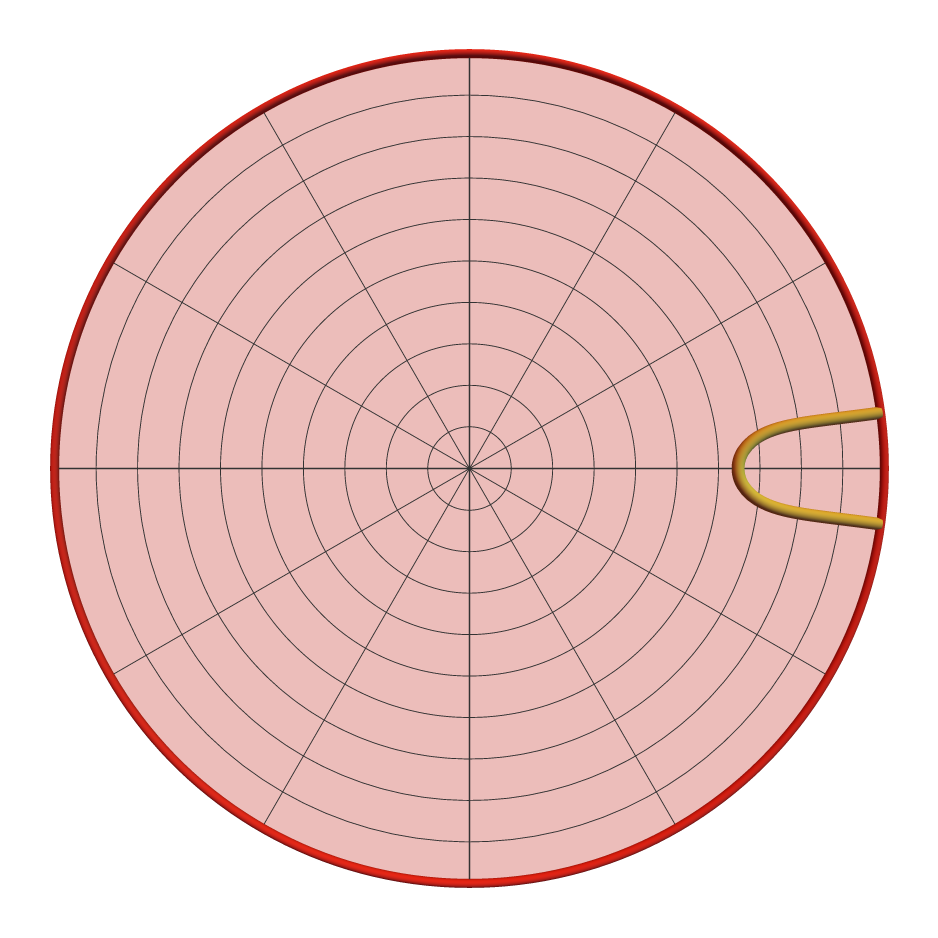}
	\includegraphics[height=5.58cm]{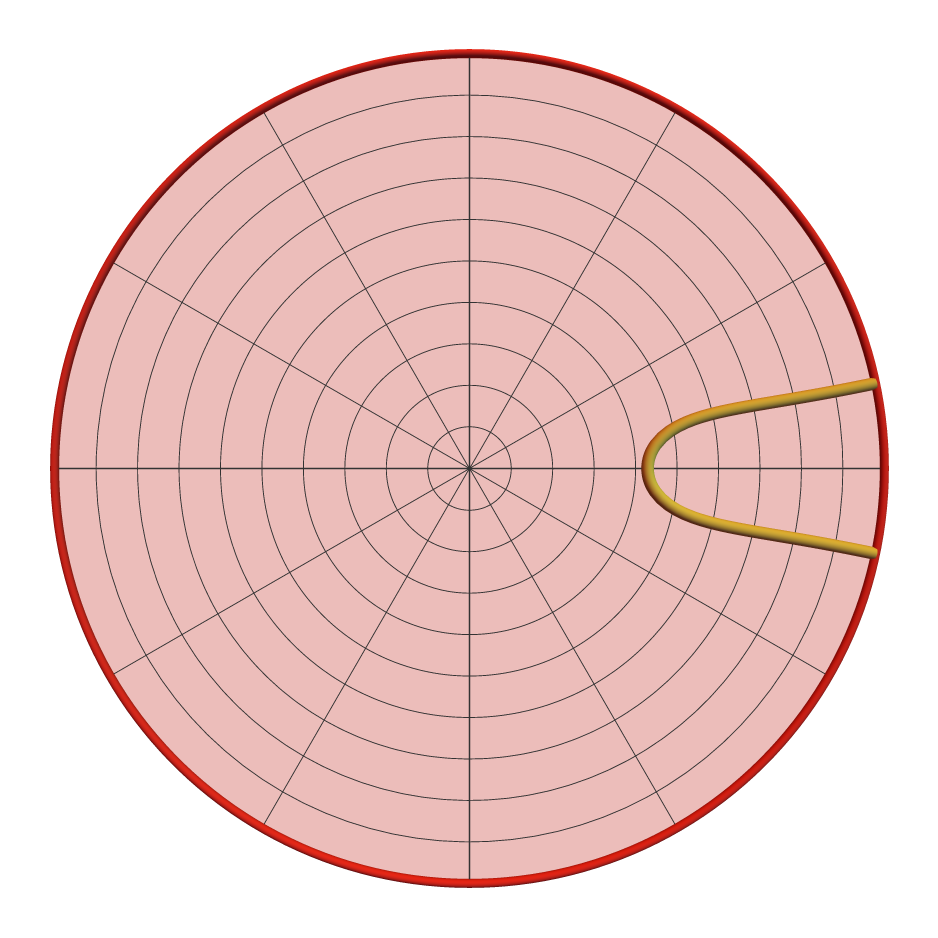}
	\includegraphics[height=5.58cm]{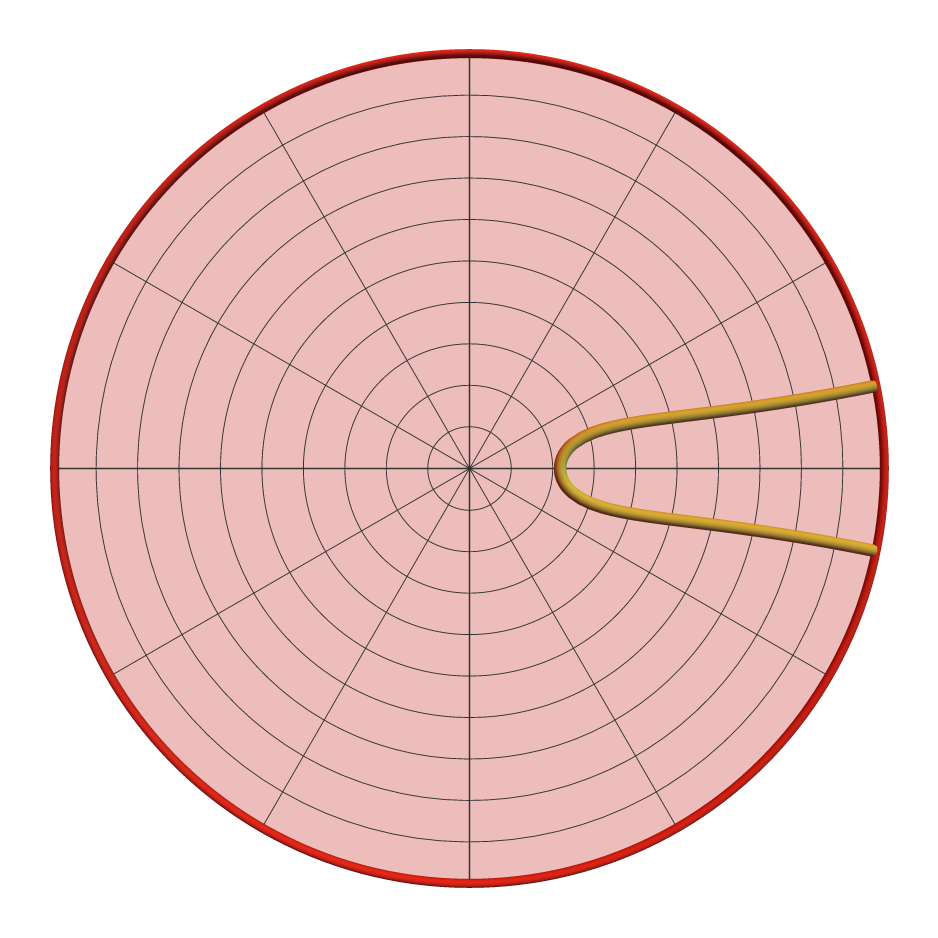}
	\caption{\small{Three $n$-catenaries for $n=2$ in the hyperbolic plane $\mathbb{H}^2$ corresponding to the values $c=1$, $c=2$, and $c=4$, respectively. They are represented in the Poincar\'e disc model.}}
	\label{F2}
\end{figure}

\begin{figure}[h!]
	\centering
	\includegraphics[height=5.58cm]{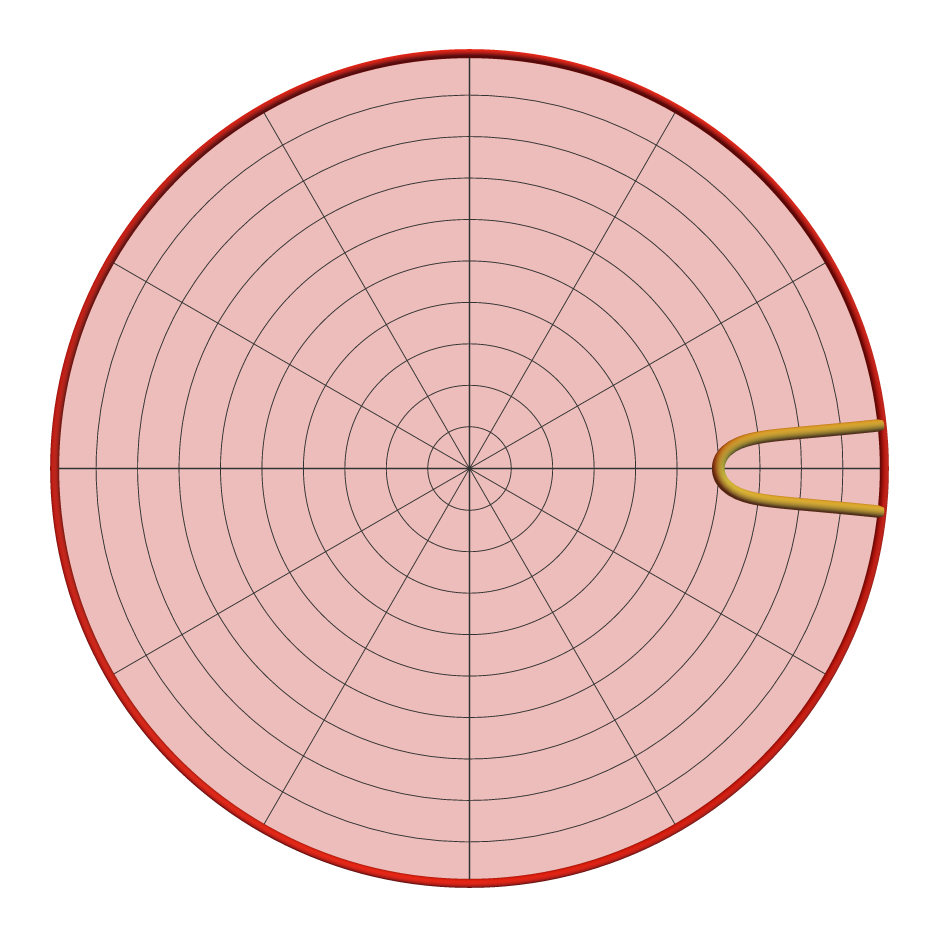}
	\includegraphics[height=5.58cm]{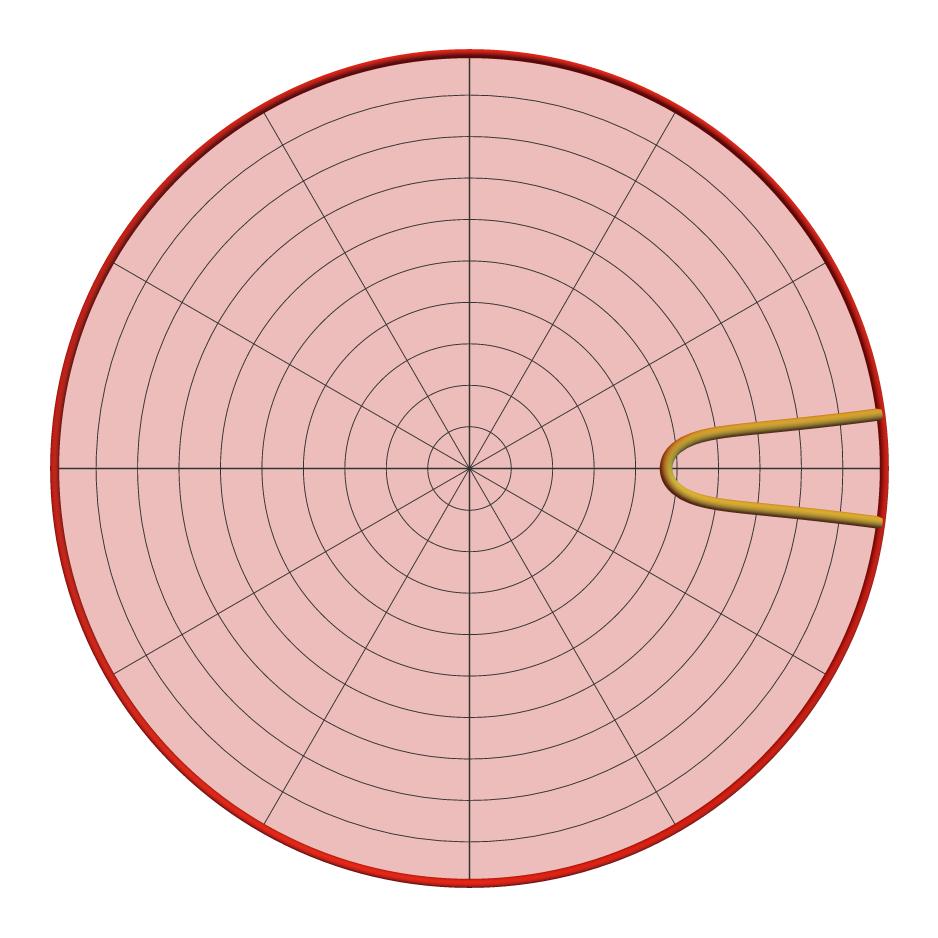}
	\includegraphics[height=5.58cm]{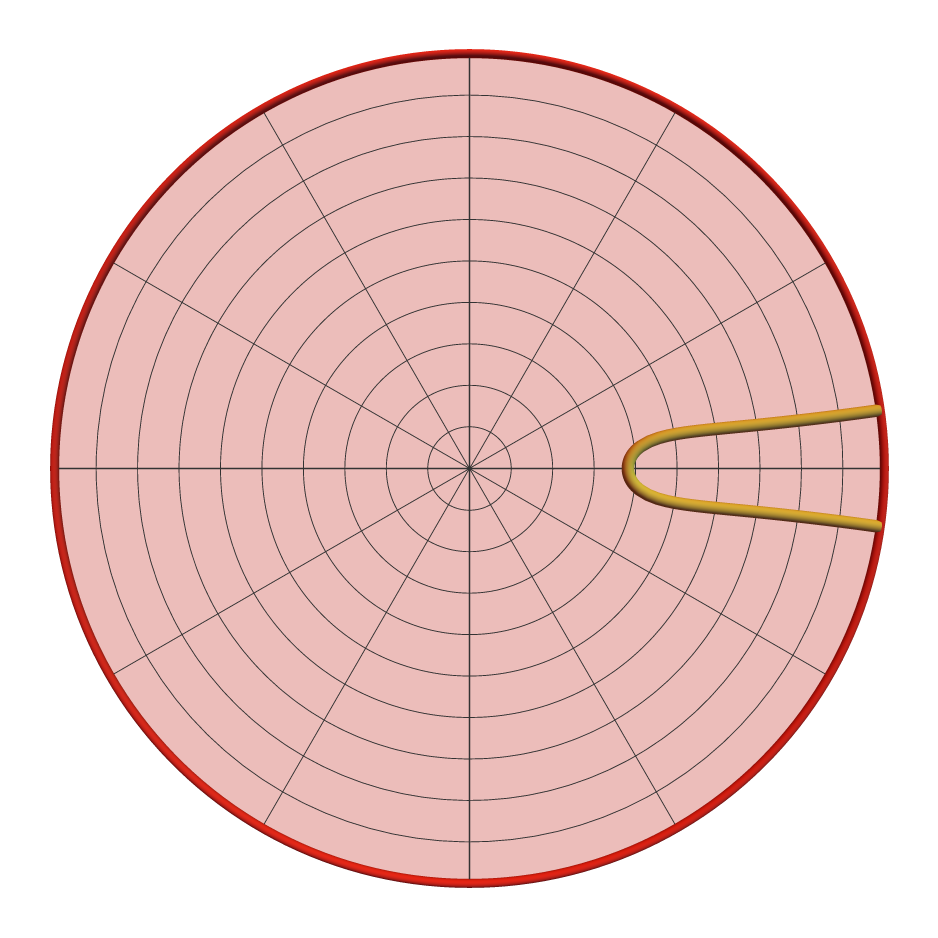}
	\caption{\small{Three $n$-catenaries for $n=3$ in the hyperbolic plane $\mathbb{H}^2$ corresponding to the values $c=2$, $c=3$, and $c=4$, respectively. They are represented in the Poincar\'e disc model.}}
	\label{F3}
\end{figure}

In Figures \ref{F1}-\ref{F3} we illustrate several $n$-catenaries ($n=1$ in Figure \ref{F1}, $n=2$ in Figure \ref{F2}, and $n=3$ in Figure \ref{F3}) for different values of $c>0$. These figures show the geometric properties of $n$-catenaries and highlight their kinematics as the parameter $c>0$ varies. In fact, let $\{\gamma_c\}$, $c>0$, be the family of $n$-catenaries for $n$ fixed. Then, as $c$ increases, the maximum curvature $\alpha^n$ increases monotonically from $2n-1$ to $\infty$ (one can prove this by implicitly differentiating the relation $-\alpha^{2n}+c^2\alpha+(2n-1)^2=0$ and understanding the asymptotic behavior of $\alpha$ when $c\to 0^+$ and $c\to\infty$). At the same time, the (Euclidean) distance $\Phi_n$ decreases monotonically\footnote{Since $\Phi_n'(c)<0$ for every $c\in(0,\infty)$, the function $\Phi_n:(0,\infty)\longmapsto (0,1)$ is bijective. Therefore, the parameter $c>0$ essentially determines this distance.} from $1$ to $0$ (this can be checked by explicitly differentiating $\Phi_n(c)$ and taking the limits as $c\to 0^+$ and $c\to \infty$). Moreover, the (Euclidean) height $\Psi_n$ increases from $0$ until its maximum, and then decreases to $0$. We point out here that, although this property is evident from the illustrations, a formal proof will necessitate the differentiation of ${\rm arctanh}(\Psi_n)$, which is in general a hyperelliptic integral and so computing its derivative becomes challenging. For the case $n=1$, this property is shown in Corollary \ref{cor} below.

We are now in the right position to prove the main result of this section, which extends Theorem 4.3 of \cite{AGP} to higher dimensional cases. The first part of the result of Theorem \ref{t1} was first stated, without proof, in \cite{P}.

\begin{theorem}\label{t1} Let $M$ be a non-totally geodesic rotational hypersurface in $\mathbb{H}^{2n+1}$. Then, $M$ is minimal if and only if its generating curve $\gamma$, locally, satisfies the Euler--Lagrange equation \eqref{ODE} associated to the $p$-energy functional
	$$\mathbf{\Theta}_p(\gamma)=\int_\gamma \kappa^p\,ds\,,$$
for $p=(2n-1)/(2n)\in[1/2,1)$ and $a>0$.

Furthermore, the rotational hypersurface $M$ is the warped product hypersurface $\gamma\times_f \mathbb{S}^{2n-1}$, where the warping function $f$ is given by
$$f=\frac{2n-1}{c\,\kappa^{1/(2n)}}=\frac{2n-1}{c\,\sqrt{u}}\,,$$
for $c=2n\sqrt{a}>0$. (Hence, the metric on $M$ is $ds^2+f^2(s)\overline{g}$, where $\overline{g}$ is the round metric on $\mathbb{S}^{2n-1}$.)
\end{theorem}
\textit{Proof.} Let $M$ be a non-totally geodesic minimal rotational hypersurface in $\mathbb{H}^{2n+1}$. Since $M$ is non-totally geodesic and minimal, the curvature of its generating curve $\gamma$ is non-constant. Hence, locally, by the inverse function theorem, we can write the arc length parameter $s$ of $\gamma$ as a function of its curvature $\kappa$. 

Let $\gamma(s)=(x_1(s),0,x_3(s),z(s))$ and define the function $P(\kappa)$ to be a primitive of $\dot{P}(\kappa)$ so that
\begin{equation}\label{x1}
	x_1(s)=x_1(s(\kappa))=\frac{\dot{P}(\kappa)}{\sqrt{a}}\,,
\end{equation}
for some constant $a>0$. (The upper dot represents the derivative with respect to $\kappa$.)

From \eqref{x1}, the curvature $\kappa$ of the arc length parameterized curve $\gamma(s)$ locally coincides with that of
\begin{equation}\label{gammatilde}
	\widetilde{\gamma}(s)=\frac{1}{\sqrt{a}}\left(\dot{P},\sqrt{a+\dot{P}^2}\,\sinh\widetilde{v}(s),\sqrt{a+\dot{P}^2}\,\cosh\widetilde{v}(s)\right),
\end{equation}
for 
$$\widetilde{v}(s)=-\sqrt{a}\int \frac{\kappa\dot{P}-P}{a+\dot{P}^2}\,ds\,.$$
Therefore, by the fundamental theorem of (hyperbolic) curves, both curves are the same (possibly after a suitable rigid motion).

The curve $\widetilde{\gamma}$ given in \eqref{gammatilde} is, up to rigid motions, the only critical curve with non-constant curvature for the general curvature-dependent functional
$$\mathbf{\Theta}(\gamma)=\int_\gamma P(\kappa)\,ds\,,$$
acting on hyperbolic curves, and $a>0$ fixed (see, for instance, \cite{AP} and references therein). In particular, its curvature $\kappa$ satisfies the (first integral of the) associated Euler--Lagrange equation, namely,
\begin{equation}\label{FI}
	\dot{P}_s^2+\left(\kappa\dot{P}-P\right)^2-\dot{P}^2=a\,.
\end{equation}
This equation was computed, for example, in \cite{AP} adapting the computations of \cite{LS} to the general case. It also follows from the condition that $\widetilde{\gamma}$ given in \eqref{gammatilde} is parameterized by the arc length.

It only remains to find the specific function $P(\kappa)$. As explained in Section 2, if $M$ is a rotational minimal hypersurface of $\mathbb{H}^{2n+1}$, then \eqref{pk} holds for the principal curvatures given in \eqref{pcurvatures}. Combining \eqref{pcurvatures} with \eqref{x1} and \eqref{FI}, we get
$$\kappa_i=-\frac{\sqrt{a+\dot{P}^2-\dot{P}_s^2}}{\dot{P}}=-\frac{\lvert \kappa\dot{P}-P\rvert}{\dot{P}}\,,$$
for $i=2,...,2n$, and
$$\kappa_1=\frac{\dot{P}_{ss}-\dot{P}}{\sqrt{a+\dot{P}^2-\dot{P}_s^2}}=\frac{\kappa P-\kappa^2\dot{P}}{\lvert \kappa\dot{P}-P\rvert}=\kappa\,,$$
for our choice of orientation of $\gamma$ (i.e., so that $\kappa\dot{P}-P<0$). The second equality in $\kappa_1$ follows from the identity
$$\dot{P}_{ss}+\dot{P}\left(\kappa^2-1\right)-\kappa P=0\,,$$
obtained differentiating \eqref{FI} with respect to $s$. The relation \eqref{pk} then reads
$$\kappa_2=\frac{\kappa\dot{P}-P}{\dot{P}}=-\frac{\kappa}{2n-1}\,,$$
which gives rise to a first-order ordinary differential equation for $P(\kappa)$ in separable variables. Solving this equation we deduce that
$$P(\kappa)=c\,\kappa^p\,,$$
for $p=(2n-1)/(2n)$. The multiplicative factor $c$ is irrelevant, so we take it to be one. This finishes the forward implication.

For the converse, let $M$ be a non-totally geodesic rotational hypersurface in $\mathbb{H}^{2n+1}$ whose generating curve $\gamma$ satisfies the Euler--Lagrange equation associated to $\mathbf{\Theta}_p$ for $a>0$ and $p=(2n-1)/(2n)$. Then, $\gamma$ is a $n$-catenary which can be parameterized as \eqref{param}, where the normalized curvature $u=\kappa^{1/n}$ is a solution of \eqref{ODE} for $a>0$. It is then a straightforward computation to verify that \eqref{pk} holds and so $M$ is minimal.

Once the above equivalence is established the second statement of the result is a direct consequence of the rotational invariance of the hypersurface $M$ and the explicit parameterization \eqref{param} of the generating curve. \hfill$\square$

\begin{rem} The result of Theorem \ref{t1} also holds for non-totally geodesic rotational hypersurfaces $M$ in $\mathbb{H}^{2n}$. Indeed, regardless of $m$ being even or odd, a non-totally geodesic (spherical) rotational hypersurface $M$ in $\mathbb{H}^m$ is minimal if and only if its generating curve satisfies the Euler--Lagrange equation associated to $\mathbf{\Theta}_p$ for $p=(m-2)/(m-1)$ (and space-like momentum).
\end{rem}

\section{Renormalized Area of Catenoids in $\mathbb{H}^3$}

Let $M_c$, $c\in(0,\infty)$, be a catenoid in $\mathbb{H}^3$. From Theorem \ref{t1} we deduce that $M_c$ is the warped product surface $\gamma_c\times_f\mathbb{S}^1$, where the generating curve $\gamma_c$ is a $1$-catenary (also known as a $1/2$-elastic curve) in $\mathbb{H}^2$, $f=(c\sqrt{u})^{-1}$ and $u=\kappa$ is the curvature of $\gamma_c$ which satisfies the Euler--Lagrange equation \eqref{ODE} for $n=1$. For this particular dimension, the ordinary differential equation \eqref{ODE} can be explicitly solved as (c.f., Corollary 3.3 of \cite{AGP})
\begin{equation}\label{curvature}
	u(s)=\frac{2}{\sqrt{c^4+4}\cosh(2s)-c^2}\,,
\end{equation}
where $s\in\mathbb{R}$ is the arc length parameter of $\gamma_c$ satisfying $u(0)=\alpha=(c^2+\sqrt{c^4+4})/2$.

Let $\theta$ be the parameter of $\mathbb{S}^1$. Then, $M_c\cong \gamma_c\times_f \mathbb{S}^1$ can be parameterized as
$$X_c(s,\theta)=\mathcal{R}(\theta)\cdot\gamma_c(s)\subset \mathbb{H}^3\subset\mathbb{L}^4\,,$$
where $\mathcal{R}(\theta)$ is the $4\times 4$ orthogonal matrix
$$\mathcal{R}(\theta)=\begin{pmatrix} \cos\theta & -\sin\theta & 0 & 0 \\ \sin\theta & \cos\theta & 0 & 0 \\ 0 & 0 & 1 & 0 \\ 0 & 0 & 0 & 1\end{pmatrix},$$
and $\gamma_c(s)$ is the $1$-catenary whose arc-length parameterization is given in \eqref{param}. Expanding, we get
$$X_c(s,\theta)=\frac{1}{c\sqrt{u}}\left(\cos\theta,\sin\theta,\sqrt{1+c^2u}\sinh v(s),\sqrt{1+c^2u}\cosh v(s)\right),$$
where the (hyperbolic) angular progression is given by
$$v(s)=-c\int_0^s \frac{u^{3/2}}{1+c^2u}\,d\sigma\,.$$
Observe that the metric induced on the catenoid $M_c$ from $\mathbb{L}^4$ is, precisely, $ds^2+(c^2 u)^{-1}d\theta^2$, which coincides with that of the warped product surface $\gamma_c\times_f\mathbb{S}^1$.

In Figure \ref{FC}, we use the above parameterization and the identification \eqref{ident} of Section 2, to show three catenoids $M_c$ in the Poincar\'e ball model for $\mathbb{H}^3$.

\begin{figure}[h!]
	\centering
	\includegraphics[height=5.5cm]{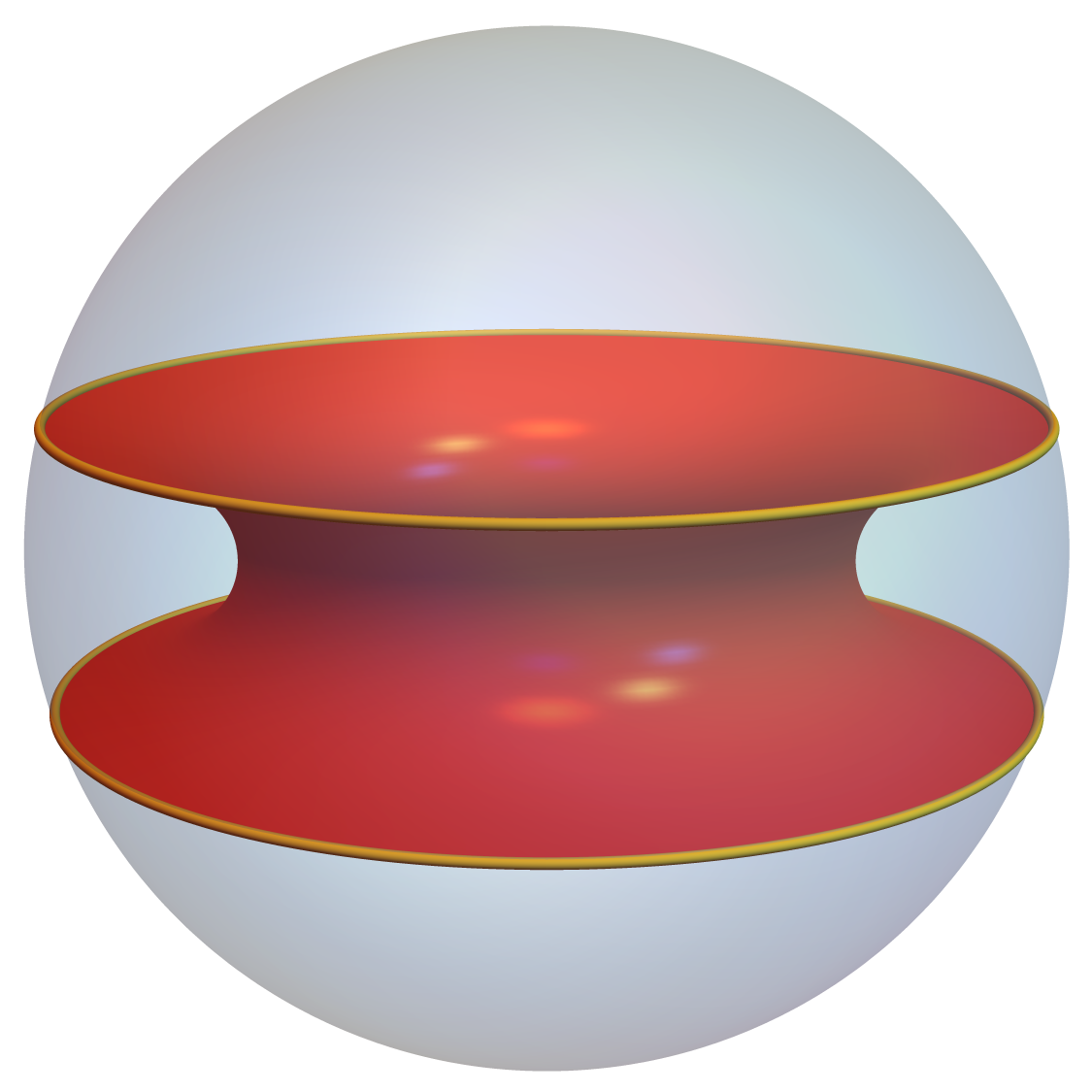}\,
	\includegraphics[height=5.5cm]{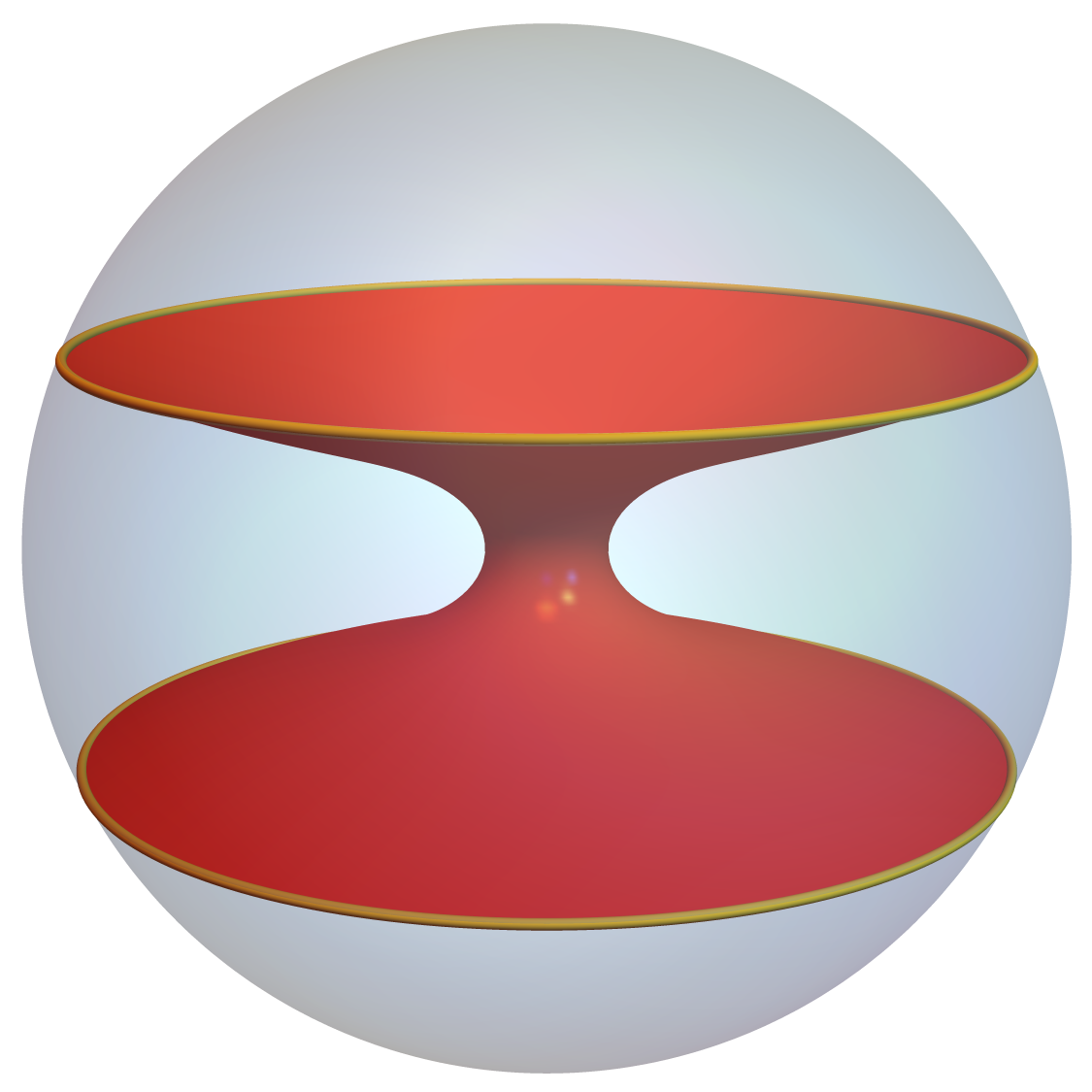}\,
	\includegraphics[height=5.5cm]{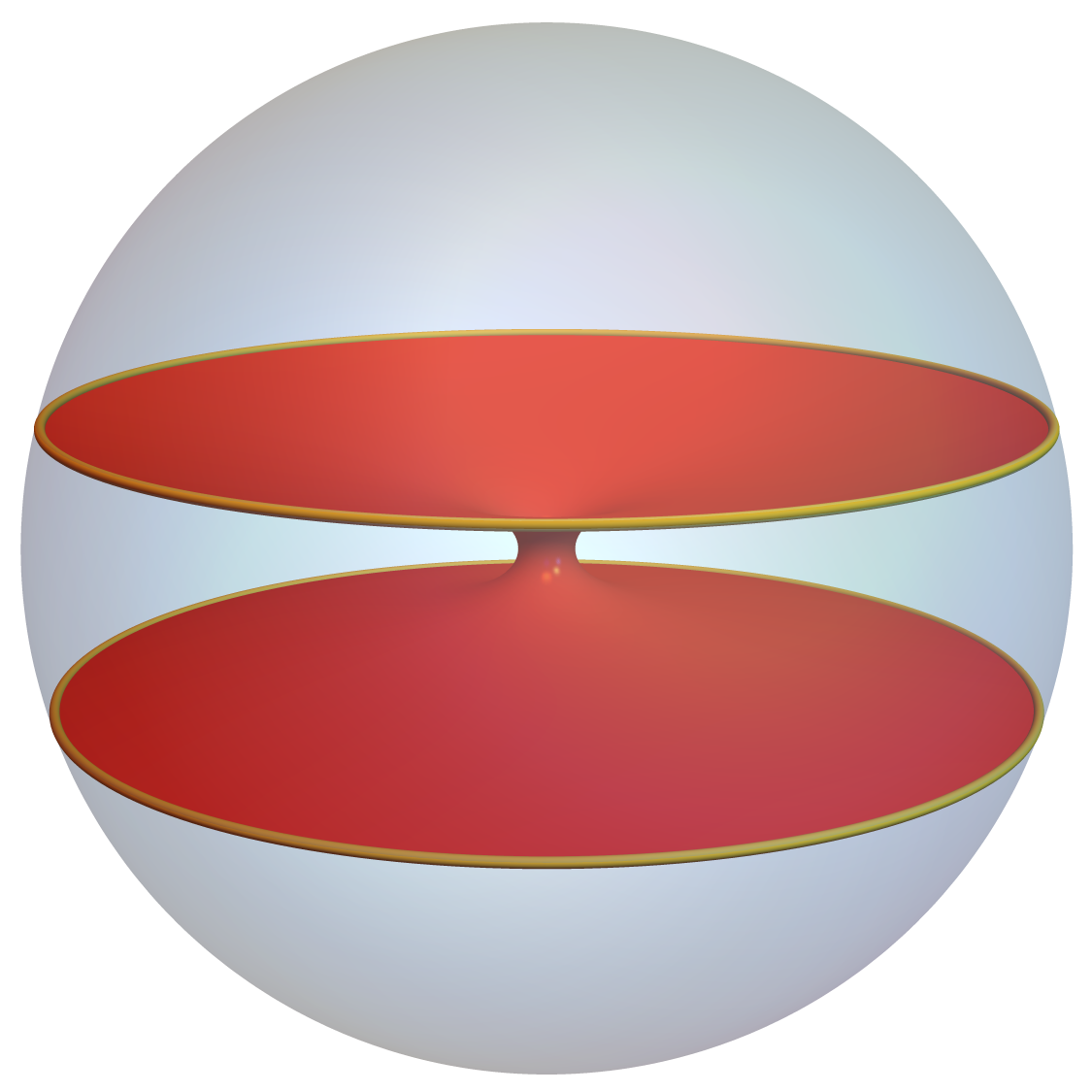}
	\caption{\small{Three catenoids $M_c$ in the Poincar\'e ball model for $\mathbb{H}^3$ for different values of the parameter $c>0$. Their corresponding generating curves are the $1$-catenaries shown in Figure \ref{F1}, with parameters $c=0.5$, $c=2$, and $c=3$, respectively.}}
	\label{FC}
\end{figure}

We will next compute the renormalized area $\mathcal{A}_R$ of the catenoid $M_c$, $c>0$, employing \eqref{formulak} and the variational characterization of Theorem \ref{t1}. The renormalized area of $M_c$ will be expressed as a linear combination of standard complete elliptic integrals of the first and second kind, denoted by $K$ and $E$, respectively. Thus, we first briefly recall here that the complete elliptic integrals of the first and second kind are defined, respectively, by
\begin{equation*}
	K(\zeta)=\int_0^1 \frac{dx}{\sqrt{(1-x^2)(1-\zeta^2x^2)}}\,,\quad\quad\quad
	E(\zeta)=\int_0^1 \frac{\sqrt{1-\zeta^2x^2}}{\sqrt{1-x^2}}\,dx\,.
\end{equation*}
Note that $K$ is a monotonically increasing function from $\pi/2$ as $\zeta\to 0^+$ to $\infty$ when $\zeta\to 1^-$; while $E$ monotonically decreases from $\pi/2$ as $\zeta\to 0^+$ to $1$ as $\zeta\to 1^-$.

The next result gives an explicit expression of the renormalized area $\mathcal{A}_R$ of catenoids $M_c$, $c\in(0,\infty)$, in $\mathbb{H}^3$.

\begin{theorem}\label{t2} Let $M_c$, $c\in(0,\infty)$, be a non-totally geodesic rotational minimal surface in $\mathbb{H}^3$. The renormalized area of $M_c$ can be expressed in terms of standard complete elliptic integrals as
	\begin{equation}\label{renA}
		\mathcal{A}_R(M_c)=\frac{4\pi}{\sqrt{2\zeta^2-1}}\left(\left(1-\zeta^2\right)K(\zeta)-E(\zeta)\right),
	\end{equation}
where $\zeta\in(\sqrt{2}/2,1)$ is given in terms of the parameter $c>0$ by
$$\zeta=\sqrt{\frac{c^2+\sqrt{c^4+4}}{2\sqrt{c^4+4}}}\,.$$

Moreover, the renormalized area $\mathcal{A}_R(M_c)$ has a unique critical point (a maximum) at the value $\zeta_+\in(\sqrt{2}/2,1)$ such that $K(\zeta_+)=2E(\zeta_+)$. Hence, $\mathcal{A}_R(M_c)$ increases from $-\infty$, when $c\to 0^+$, to its maximum 
$$\mathcal{A}_{R}^+=-4\pi\sqrt{2\zeta_+^2-1}\,E(\zeta_+)\in\left(-4\pi,-\pi^2\right),$$ 
and then decreases approaching $-4\pi$, as $c\to\infty$.
\end{theorem}
\textit{Proof.} Let $M_c$, $c\in(0,\infty)$, be a catenoid in $\mathbb{H}^3$ and denote by $M_\epsilon=M_c\cap\{z\leq 1/\epsilon\}$, for $\epsilon>0$ sufficiently small, the truncated catenoid. Then, $M_\epsilon$ is compact and embedded in $\mathbb{H}^3$. Applying Proposition \ref{prop2} for $n=1$ to $M_\epsilon$, we get
$$\mathcal{A}(M_\epsilon)=-\int_{M_\epsilon}\lambda\,dA-\int_{M_\epsilon} \kappa^2\,dA\,.$$
In addition, since for surfaces ($n=1$), $\lambda=K$ is just the Gaussian curvature, we can employ the classical Gauss--Bonnet theorem
$$\int_{M_\epsilon} K\,dA=2\pi\chi(M_\epsilon)-\oint_{\partial M_\epsilon}\kappa_g\,ds\,,$$
in the first integral above, obtaining
$$\mathcal{A}(M_\epsilon)=-2\pi\chi(M_\epsilon)+\oint_{\partial M_\epsilon} \kappa_g\,ds-\int_{M_\epsilon}\kappa^2\,dA=\oint_{\partial M_\epsilon}\kappa_g\,ds-\int_{M_\epsilon}\kappa^2\,dA\,.$$ 
Observe that the second equality follows from $\chi(M_\epsilon)=0$, since $M_\epsilon$ is a topological annulus. Moreover, from (3.11) of \cite{AM}, which shows that the total geodesic curvature has no constant term when considering its expansion in terms of $\epsilon>0$, we deduce that
\begin{equation}\label{test1}
	\mathcal{A}_R(M_c)=-\int_{M_c}\kappa^2\,dA=-\frac{2\pi}{c}\int_{\gamma_c} u^{3/2}\,ds=-\frac{2\pi}{c}\int_0^\alpha \frac{u}{\sqrt{u(-u^2+c^2u+1)}}\,du\,.
\end{equation}
The second equality above follows from $u=\kappa$ and $dA=(c^2u)^{-1/2}d\theta ds$, which is a consequence of the induced metric on $M_c\cong \gamma_c\times_f \mathbb{S}^1$. On the other hand, the last equality holds after using the Euler--Lagrange equation \eqref{ODE} to make a change of variable. Recall that the curve $\gamma_c$ is a $1$-catenary and so \eqref{ODE} must be satisfied for $n=1$. Furthermore, $\gamma_c$ is symmetric with respect to $\{x_3=0\}$ (Part (iii) of Proposition \ref{prop}) with maximum curvature $\alpha$. In this particular case, it can be explicitly computed as $\alpha=(c^2+\sqrt{c^4+4})/2$.

Finally, the expression of the renormalized area $\mathcal{A}_R(M_c)$ in terms of $K(\zeta)$ and $E(\zeta)$ follows from well known relations between elliptic integrals (see, for instance, formula 3.132-5 of \cite{GR} or 235.00 of \cite{BF}).

We will next understand the right-hand side of \eqref{renA}. We begin computing the asymptotic behavior. From the Legendre's relation for elliptic integrals at $\zeta=\sqrt{2}/2$, we deduce that
$$\frac{1}{2}K\left(\sqrt{2}/2\right)-E\left(\sqrt{2}/2\right)=\frac{-\pi}{4K\left(\sqrt{2}/2\right)}\,,$$
which combined with $K(\zeta)$ being strictly increasing from $1$ as $\zeta\to 0^+$, shows that the parenthesis on the right-hand side of \eqref{renA} is finite and negative. Hence,
$$\lim_{\zeta\to\sqrt{2}/2^+}\mathcal{A}_R(M_c)=-\infty\,.$$
For the limit as $\zeta\to 1^-$, we recall that $K(\zeta)\simeq \log(4)-\log(1-\zeta)/2$. Therefore, $(1-\zeta^2)K(\zeta)\to 0$ as $\zeta\to 1^-$. Together with $E(1)=1$, we then get
$$\lim_{\zeta\to 1^-}\mathcal{A}_R(M_c)=-4\pi\,.$$
Now, we will compute the derivative. Observe first that as $c>0$ increases, the parameter $\zeta$ increases monotonically from $\sqrt{2}/2$ to $1$. Hence, it is sufficient to compute the derivative of \eqref{renA} in terms of $\zeta$. Using the well known derivatives of $K(\zeta)$ and $E(\zeta)$ (see, for instance, formulas 8.123-2 and 8.123-4 of \cite{GR}) we get after straightforward simplifications
$$\frac{d}{d\zeta}\mathcal{A}_R(M_c)=\frac{-4\pi\zeta}{(2\zeta^2-1)^{3/2}}\left(K(\zeta)-2E(\zeta)\right).$$
As $\zeta\to\sqrt{2}/2^+$, we can use once again Legendre's identity to see that the above derivative approaches $+\infty$; while as $\zeta\to 1^-$ we can use once again the asymptotic behavior $K(\zeta)\simeq \log(4)-\log(1-\zeta)/2$ to check that it approaches $-\infty$. This combined with $K$ being increasing and $E$ being decreasing shows that $\mathcal{A}_R(M_c)$ has only one critical point and, hence, $\mathcal{A}_R(M_c)$ increases until it attains the maximum at the value $\zeta_+\in(\sqrt{2}/2,1)$ such that $K(\zeta_+)=2E(\zeta_+)$. Then, $\mathcal{A}_R(M_c)$ decreases approaching $-4\pi$ from above. It only remains to prove that the maximum of $\mathcal{A}_R(M_c)$, namely $\mathcal{A}_R^+$, is smaller than $-\pi^2$. Assume to the contrary that $\mathcal{A}_R^+\geq -\pi^2$. One can then double the catenoid $M_c$ obtaining a torus. From Proposition 8.1 of \cite{AM}, this torus has Willmore energy $-2\mathcal{A}_R^+\leq 2\pi^2$, which contradicts the resolution of the Willmore conjecture \cite{MN}. This finishes the proof. \hfill$\square$

\begin{rem}\label{rem} The renormalized area $\mathcal{A}_R$ of catenoids $M_c$, $c\in(0,\infty)$, in $\mathbb{H}^3$ can also be computed from the second integral in \eqref{test1} and the explicit expression of the curvature of $1$-catenaries \eqref{curvature}. In fact, we get after the change of variable $x=2s$ that
	$$\mathcal{A}_R(M_c)=-\frac{2\pi}{c}\int_{\gamma_c} u^{3/2}\,ds=-\frac{4\sqrt{2}\,\pi}{c}\int_0^\infty \frac{dx}{\left(\sqrt{c^4+4}\,\cosh x-c^2\right)^{3/2}}\,.$$
The improper integral above can be explicitly described in terms of standard complete elliptic integrals of the first and second kind (for details, see for instance, 2.464-43 of \cite{GR}). After some straightforward manipulations we recover \eqref{renA}, as expected.
\end{rem}

We highlight here that the renormalized area of minimal surfaces in $\mathbb{H}^3$ is always negative. From Proposition 8.1 of \cite{AM}, the renormalized area of minimal surfaces in $\mathbb{H}^3$ coincides, up to a coefficient $-2$, with the Willmore energy of its double regarded as a surface in $\mathbb{R}^3$ and, hence, it must be negative (see also the expressions of $\mathcal{A}_R(M_c)$ given in \eqref{test1} and in Remark \ref{rem}). However, the result of Theorem \ref{t2} shows that the renormalized area of minimal surfaces is not bounded below. Indeed, the behavior of \eqref{renA} proves that for every value $m$ in $(-\infty,-4\pi]$, there exists a unique catenoid $M_c$ with $\mathcal{A}_R(M_c)=m$, while for every $m\in(-4\pi,\mathcal{A}_R^+)$, there are two different catenoids with the same renormalized area $m$. 

\begin{figure}[h!]
	\begin{center}
		\includegraphics[height=4cm,width=7.5cm]{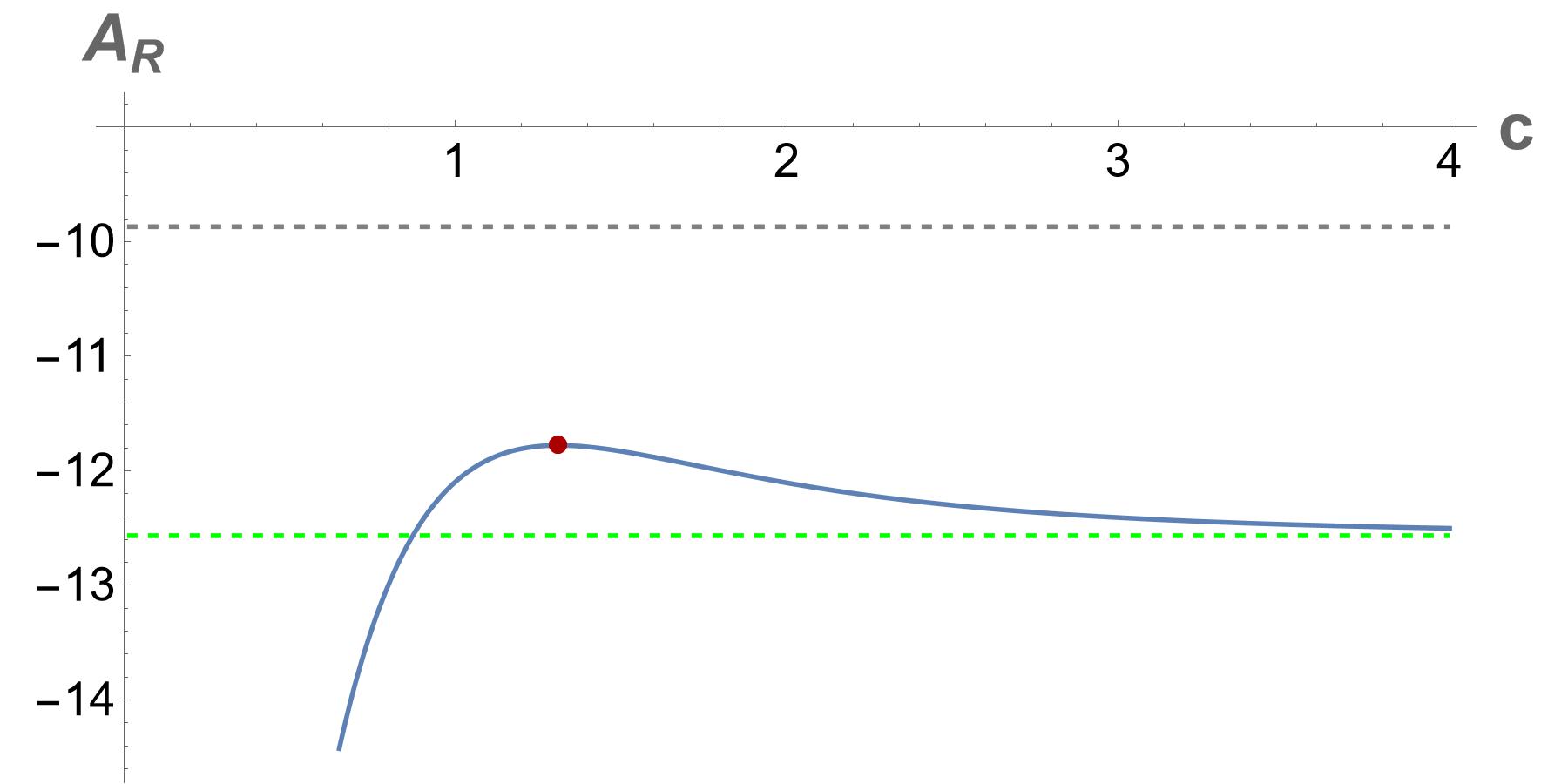}
	\end{center} 
	\caption{\small{Graph of the renormalized area $\mathcal{A}_R$ of the catenoids $M_c$ in $\mathbb{H}^3$ as a function of the parameter $c\in(0,\infty)$. The green dashed line represents twice the renormalized area of $\mathbb{H}^2$, i.e., $-4\pi\simeq -12.5664$. The gray dashed line is the upper bound $-\pi^2\simeq -9.8696$ for the renormalized area of minimal annuli. The red dot is the maximum of $\mathcal{A}_R(M_c)$ which is, approximately, $-11.7778$ and is attained at $c_+\simeq 1.3118$.}}
	\label{G2}
\end{figure}

Furthermore, note that the value $-4\pi$ arising as the limit of \eqref{renA} is not coincidental, it is in fact twice the renormalized area of the hyperbolic plane $\mathbb{H}^2$. The fact that the limit of $\mathcal{A}_R(M_c)$ as $c\to\infty$ is precisely $-4\pi$ is expectable due to the geometric behavior of $1$-catenaries as $c\to\infty$ (explained in the previous section). The catenoids $M_c$ approach a double cover of the equator in the ball model for $\mathbb{H}^3$, i.e., twice the hyperbolic plane $\mathbb{H}^2$. In Figure \ref{G2}, we illustrate the graph of the renormalized area $\mathcal{A}_R$ of the catenoids $M_c$, $c\in(0,\infty)$.

Finally, we also observe that the unique catenoid attaining the maximum of the renormalized area $\mathcal{A}_R(M_c)$ has a geometric particularity, which we show in the following corollary.

\begin{corollary}\label{cor} The maximum value of the renormalized area of the family of catenoids $\{M_c\}$, $c>0$, is attained at the catenoid with maximum height $\Psi_1$.
\end{corollary}
\textit{Proof.} As noted in the proof of Theorem \ref{t2}, the maximum of the renormalized areas $\mathcal{A}_R$ of the catenoids $M_c$, $c>0$, is attained at the value $\zeta_+\in(\sqrt{2}/2,1)$ such that $K(\zeta_+)=2E(\zeta_+)$. 

We will now compute at what value of $c>0$ (equivalently, $\zeta\in(\sqrt{2}/2,1)$) the catenoids $M_c$ attain the maximum height and check that it is precisely $\zeta_+$. The height of the catenoids $M_c$ is just the height of their generating curves, which was introduced in Part (v) of Proposition \ref{prop} (recall that $n=1$). Since the hyperbolic tangent is always increasing, it is enough to understand the integral
$$\widehat{\Psi}_1(c)={\rm arctanh}\left(\Psi_1(c)\right)=\frac{c}{2}\int_0^\alpha \frac{u}{(c^2u+1)\sqrt{u(-u^2+c^2u+1)}}\,du\,,$$
where $\alpha=(c^2+\sqrt{c^4+4})/2$ is the only positive root of the polynomial $Q_{1,c}(u)=-u^2+c^2u+1$. Taking the $c^2$ out of the denominator and adding and subtracting $1/c^2$ to the numerator, we decompose $\widehat{\Psi}_1$ as
$$\widehat{\Psi}_1(c)=\frac{1}{2c}\int_0^\alpha \frac{du}{\sqrt{u\,Q_{1,c}(u)}}-\frac{1}{2c^3}\int_0^\alpha \frac{du}{(u+1/c^2)\sqrt{u\,Q_{1,c}(u)}}\,.$$
The first integral above is a multiple of the standard complete elliptic integral of the first kind $K$ (see, for instance, 3.131-6 of \cite{GR}); while the second one is a multiple of the standard complete elliptic integral of the third kind (see, for instance, 3.137-6 of \cite{GR}), which is defined as
$$\Pi(\eta,\zeta)=\int_0^1 \frac{dx}{(1-\eta x^2)\sqrt{(1-x^2)(1-\zeta^2x^2)}}\,.$$ 
More precisely, combining everything and simplifying, we obtain that
\begin{equation}\label{third} \widehat{\Psi}_1(c)=\frac{\sqrt{1-\zeta^2}}{\zeta \sqrt{2\zeta^2-1}}\left(\zeta^2 K(\zeta)-(1-\zeta^2)\Pi\left(\frac{2\zeta^2-1}{\zeta^2},\zeta\right)\right),
\end{equation}
where $\zeta\in(\sqrt{2}/2,1)$ is as in the statement of Theorem \ref{t2}, namely,
$$\zeta=\sqrt{\frac{c^2+\sqrt{c^4+4}}{2\sqrt{c^4+4}}}\,.$$
Hence, using the formulas for the derivatives of $K(\zeta)$ and $\Pi(\eta,\zeta)$ given in 8.123-2 of \cite{GR} and \cite{BF}, and simplifying, we deduce that
$$\frac{d}{d\zeta}\widehat{\Psi}_1(c)=\frac{-1}{\sqrt{1-\zeta^2}\sqrt{2\zeta^2-1}}\left(K(\zeta)-2E(\zeta)\right).$$
Consequently, the maximum of $\widehat{\Psi}_1(c)$ (and so of $\Psi_1(c)$) is also attained at $\zeta_+\in(\sqrt{2}/2,1)$ such that $K(\zeta_+)=2E(\zeta_+)$. This completes the proof. \hfill$\square$
\\

As shown in the proofs above, the maximum of $\mathcal{A}_R(M_c)$ is attained at the value of the parameter $\zeta_+\in(\sqrt{2}/2,1)$ such that $K(\zeta_+)=2E(\zeta_+)$. Even though this value cannot be explicitly obtained with standard techniques, it can be numerically approximated as $\zeta_+\simeq 0.9089$ (equivalently, $c_+\simeq 1.3118$). According to Corollary \ref{cor}, at the value $c_+$, the corresponding catenoid $M_{c_+}$ attains the maximum height (see Part (v) of Proposition \ref{prop} for its definition), namely, $\Psi_1(c_+)=\tanh(\widehat{\Psi}_1(c_+))$. From \eqref{third}, we compute $\widehat{\Psi}_1(c_+)\simeq 0.501143$. Observe that this value coincides with previous studies \cite{G,M,W}. Indeed, it is exactly $\rho(a_c)$ in the notation of Theorem 1.3 of \cite{W}. The catenoid $M_{c_+}$ with maximum renormalized area and maximum height is shown in the front page.

\section{Renormalized Area of Catenoids in $\mathbb{H}^5$}

Let $M_c$, $c\in(0,\infty)$, be a catenoid in $\mathbb{H}^5$. The catenoid $M_c$ can be understood as the warped product hypersurface $\gamma_c\times_f\mathbb{S}^3$, where the generating curve $\gamma_c$ is a $2$-catenary ($3/4$-elastic curve) in $\mathbb{H}^2$ and the warping function is given by
$$f(s)=\frac{3}{c\sqrt{u(s)}}\,,$$
where $u(s)=\kappa^{1/2}(s)$ is the normalized curvature. Hence, a explicit parameterization can be directly deduced from the parameterization \eqref{param} of $\gamma_c(s)$ and the standard parameterization of $\mathbb{S}^3$. In particular, the metric induced on $M_c$ from $\mathbb{H}^5\subset\mathbb{L}^6$ is $ds^2+27(c^2u)^{-3}\overline{g}$, where $\overline{g}$ is the round metric on $\mathbb{S}^3$.

We next compute the renormalized area $\mathcal{A}_R$ of the catenoids $M_c$, $c\in(0,\infty)$, in $\mathbb{H}^5$ as a hyperelliptic integral.

\begin{theorem}\label{t3} Let $M_c$, $c\in(0,\infty)$, be a non-totally geodesic rotational minimal hypersurface in $\mathbb{H}^5$. The renormalized area of $M_c$ is given by the hyperelliptic integral
	\begin{equation}\label{renA4}
		\mathcal{A}_R(M_c)=\frac{6\pi^2}{c^3}\int_0^\alpha \frac{u^6-6u^2}{\sqrt{u(-u^4+c^2u+9)}}\,du\,,
	\end{equation}
where $\alpha$ is the only positive solution of $-u^4+c^2u+9=0$.

Moreover, the renormalized area $\mathcal{A}_R$ of $M_c$ varies continuously from $-\infty$, when $c\to 0^+$, to $8\pi^2/3$, as $c\to\infty$.
\end{theorem}
\textit{Proof.} Consider the part $M_\epsilon=M_c\cap\{z\leq 1/\epsilon\}$, $\epsilon>0$ sufficiently small, of a catenoid $M_c$, $c\in(0,\infty)$, in $\mathbb{H}^5$. We apply Proposition \ref{prop2} with $n=2$ to the truncated catenoid $M_\epsilon$, to get
\begin{eqnarray*}
	\mathcal{A}(M_\epsilon)&=&\int_{M_\epsilon}\lambda^2\,dA-\frac{2}{9}\int_{M_\epsilon} \kappa^2\,dA-\frac{1}{81}\int_{M_\epsilon}\kappa^4\,dA\\
	&=&\frac{1}{12}\int_{M_\epsilon} \lvert E\rvert^2\,dA-\frac{1}{3}\oint_{\partial M_\epsilon} S\,ds-\frac{2}{9}\int_{M_\epsilon}\kappa^2\,dA-\frac{1}{81}\int_{M_\epsilon}\kappa^4\,dA\,,
\end{eqnarray*}
where $E$ is the trace-free Ricci tensor and $S$ is a boundary curvature whose explicit expression can be found, for instance, in (8) of \cite{PT} (since it will not play any role, we avoid writing it here). Observe that for the last equality, we have employed the Chern--Gauss--Bonnet formula (see, for instance, \cite{CC,CS,PT,T})
$$\int_{M_\epsilon}\lambda^2\,dA=\frac{4}{3}\pi^2\chi(M_\epsilon)-\frac{1}{24}\int_{M_\epsilon}\lvert W\rvert^2\,dA+\frac{1}{12}\int_{M_\epsilon}\lvert E\rvert^2\,dA-\frac{1}{3}\oint_{\partial M_\epsilon} S\,ds\,,$$
together with $\chi(M_\epsilon)=0=\lvert W\rvert^2$. Since $M_\epsilon$ is a part of a catenoid preserving its topology, its Euler characteristic is zero. In addition, since it is rotational, it is also locally conformally flat (see Page 704 of \cite{DCD} and references therein) and so its Weyl tensor $W$ vanishes. From Claim 4.3 of \cite{T}, the boundary integral above has no constant term when expanding it in terms of $\epsilon>0$. Moreover, from the Gauss' equation we can rewrite the squared norm of the trace-free Ricci tensor $\lvert E\rvert^2$ in terms of the second fundamental form $B$ as (c.f., (23) of \cite{PT} and (4.41) of \cite{T})
$$\lvert E\rvert^2=\lvert B^2\rvert^2-\frac{1}{4}\lvert B\rvert^4\,.$$
In turn, we deduce from \eqref{B2} (for $n=2$) that for catenoids, $\lvert B\rvert^4=16\kappa^4/9$, while
$$\lvert B^2\rvert^2=\sum_{i=1}^4 \left(\kappa_i^2\right)^2=\kappa^4+\sum_{i=2}^4 \frac{\kappa^4}{3^4}=\frac{28}{27}\,\kappa^4\,,$$
since $\kappa_1=\kappa$ and $\kappa_2=\kappa_3=\kappa_4=-\kappa/3$ (see \eqref{pk}).

Combining everything we conclude with
$$\mathcal{A}_R(M_c)=\frac{1}{27}\int_{M_c}\kappa^4\,dA-\frac{2}{9}\int_{M_c}\kappa^2\,dA=\frac{1}{27}\int_{M_c}\left(\kappa^4-6\kappa^2\right)dA\,.$$
The convergence of these integrals will be clear after the following manipulations (see also Remark \ref{convergence}). 

According to Theorem \ref{t1}, the catenoid $M_c$ is the warped product hypersurface $\gamma_c\times_f\mathbb{S}^3$, where $\gamma_c$ is a $2$-catenary, $f=3/(c\sqrt{u})$, and $u=\kappa^{1/2}$. Hence, $dA=27(c^2u)^{-3/2}d\overline{A}\,ds$. From this, and employing the Euler--Lagrange equation \eqref{ODE} to make a change of variable, we get 
$$\mathcal{A}_R(M_c)=\frac{2\pi^2}{c^3}\int_{\gamma_c}\left(u^{13/2}-6u^{5/2}\right)ds=\frac{6\pi^2}{c^3}\int_0^\alpha \frac{u^6-6u^2}{\sqrt{u(-u^4+c^2u+9)}}\,du\,,$$
where $\alpha$ is the only positive root of the polynomial $Q_{2,c}(u)=-u^4+c^2u+9$. This proves the first part of the statement. 

For the second part of the statement, we note that the right-hand side of \eqref{renA4} is a continuous function in $c\in(0,\infty)$ and compute its asymptotic behavior. We begin considering the limit $c\to 0^+$. For this, we focus on
$$I_2=\lim_{c\to 0^+} \int_0^\alpha \frac{u^6-6u^2}{\sqrt{u(-u^4+c^2u+9)}}\,du=\int_0^{\sqrt{3}}\frac{u^6-6u^2}{\sqrt{u(-u^4+9)}}\,du\,.$$
To explicitly compute the value of the integral $I_2$, we make the change of variable $u=\sqrt{3}\,x^{1/4}$, obtaining
$$I_2=\frac{3^{5/4}}{4}\int_0^1 \frac{3x^{5/8}-2x^{-3/8}}{\sqrt{1-x}}\,dx=\frac{3^{5/4}}{4}\left(3\int_0^1 x^{5/8}(1-x)^{-1/2}dx-2\int_0^1x^{-3/8}(1-x)^{-1/2}dx\right).$$
The two integrals on the right-hand side above are Beta functions $\beta(\cdot,\cdot)$, hence,
$$I_2=\frac{3^{5/4}}{4}\left(3\,\beta\left(\frac{13}{8},\frac{1}{2}\right)-6\,\beta\left(\frac{5}{8},\frac{1}{2}\right)\right)=-\frac{2\cdot 3^{1/4}\sqrt{\pi}\,\Gamma\left(\frac{5}{8}\right)}{\Gamma\left(\frac{1}{8}\right)}\,,$$
where we have used the standard relation between the Beta function and the Gamma function $\Gamma(\cdot)$ and that $\Gamma(1/2)=\sqrt{\pi}$. Therefore, $I_2<0$ is finite and so
$$\lim_{c\to 0^+}\mathcal{A}_R(M_c)=-\infty\,.$$

For the limit when $c\to\infty$, we first observe that the positive root $\alpha$ of $-u^4+c^2u+9$ behaves asymptotically as $\alpha\simeq c^{2/3}$. After applying the dominated convergence theorem and making the change of variable $u=c^{2/3}x$ (and after $y=x^3$), we get
$$\lim_{c\to\infty}\mathcal{A}_R(M_c)=6\pi^2\int_0^1\frac{x^5}{\sqrt{1-x^3}}\,dx=2\pi^2\int_0^1\frac{y}{\sqrt{1-y}}\,dy=\frac{8\pi^2}{3}\,,$$
since only the leading term of the numerator plays a role when $c\to\infty$. \hfill$\square$
\\

Contrary to the case of minimal surfaces in $\mathbb{H}^3$, the above result shows that the renormalized area of minimal hypersurfaces in $\mathbb{H}^5$ does not have a sign. Indeed, for the family of catenoids $\{M_c\}$, $c>0$, it varies from $-\infty$ to twice the renormalized area of the hyperbolic $4$-space $\mathbb{H}^4$ (c.f., Appendix A), namely, $2\mathcal{A}_R(\mathbb{H}^4)=8\pi^2/3$. This is in agreement with the evolution of $2$-catenaries as the parameter $c>0$ varies (see Section 3).

\begin{figure}[h!]
	\begin{center}
		\includegraphics[height=4cm,width=7.5cm]{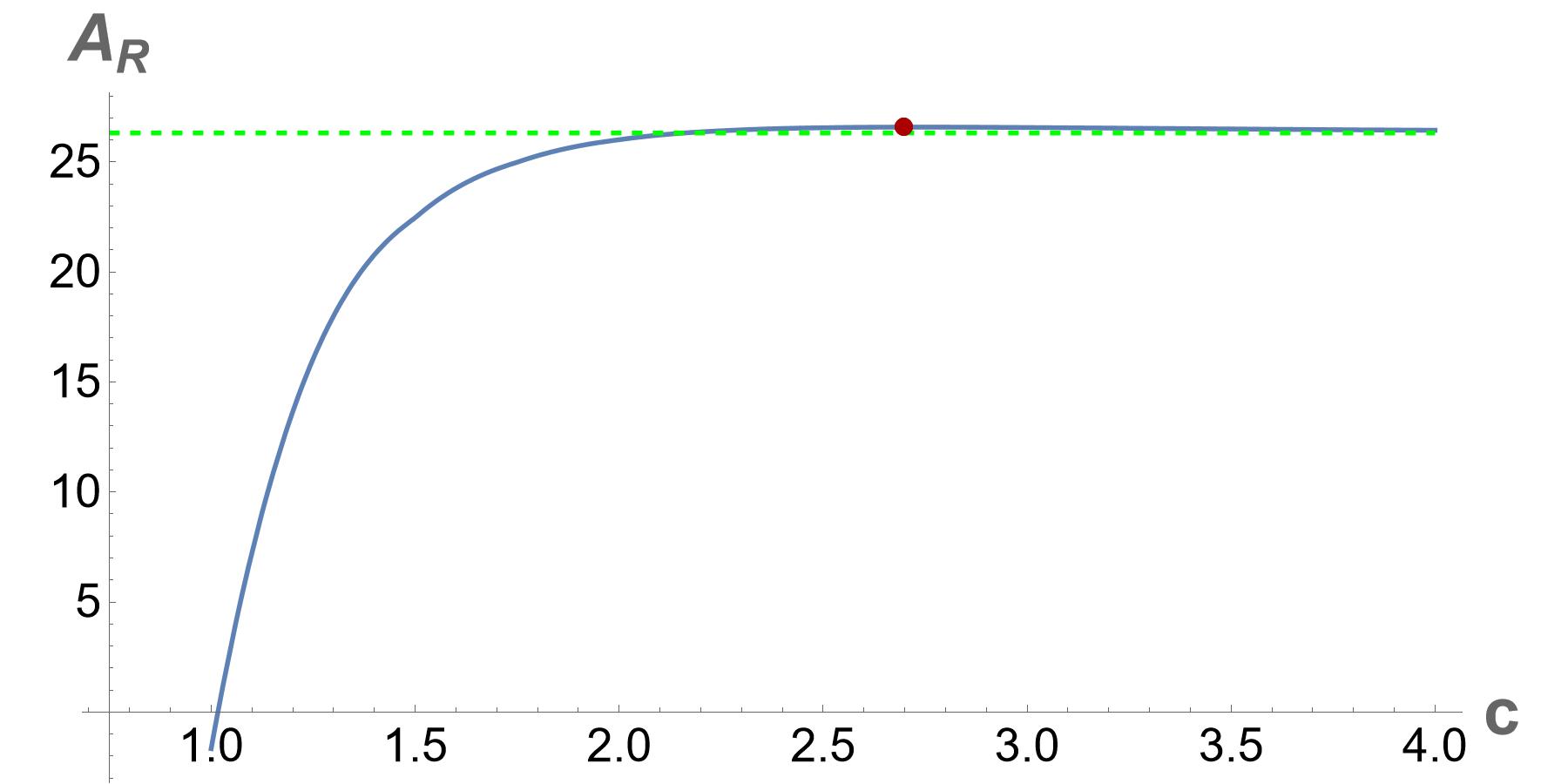}\quad\quad\quad\includegraphics[height=4cm,width=4.5cm]{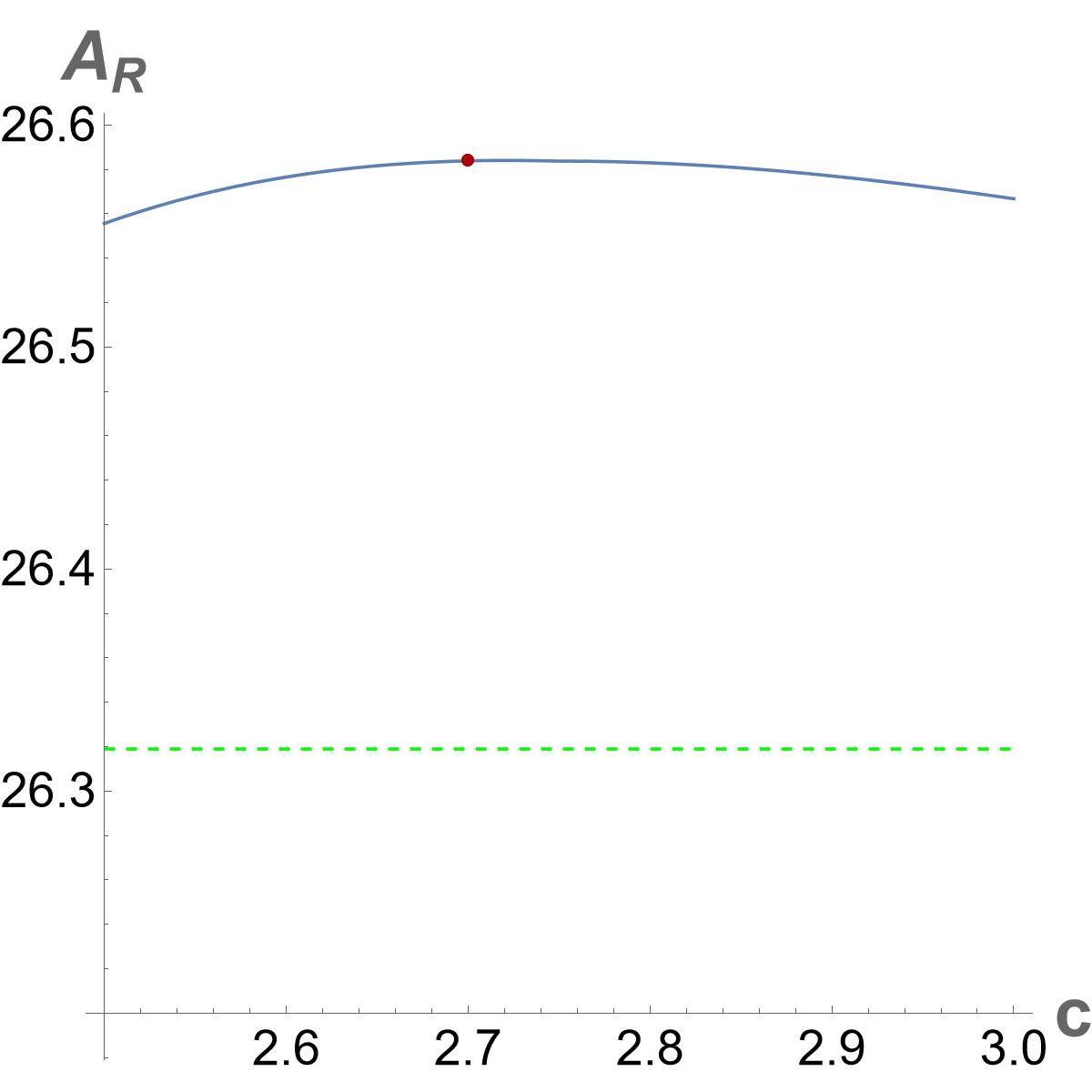}
	\end{center} 
	\caption{\small{Left: Graph of the renormalized area $\mathcal{A}_R$ of the catenoids $M_c$ in $\mathbb{H}^5$ as a function of the parameter $c\in(0,\infty)$. The green dashed line represents twice the renormalized area of $\mathbb{H}^4$, i.e., $8\pi^2/3\simeq 26.3189$. The red dot is the maximum of $\mathcal{A}_R(M_c)$ which is approximately $26.5840$, slightly larger than $8\pi^2/3$. Right: Zoom in the maximum of $\mathcal{A}_R(M_c)$.}}
	\label{G4}
\end{figure}

The geometric behavior of $\mathcal{A}_R(M_c)$ is illustrated in Figure \ref{G4} and it is analogous to that of the case of surfaces. More precisely, the renormalized area $\mathcal{A}_R(M_c)$ increases from $-\infty$ until it attains its maximum $\mathcal{A}_R^+>8\pi^2/3$ and then it decreases approaching $8\pi^2/3$ from above. Moreover, a numerical computation of the height $\Psi_2(c)$ (see Part (v) of Proposition \ref{prop}) suggests that $\mathcal{A}_R^+$ is attained at the catenoid with maximum height, in analogy with the case of surfaces.

\section{Computations in Higher Dimensions}

In this section we will generalize Theorems \ref{t2} and \ref{t3} to higher dimensional cases computing the renormalized area of catenoids $M_c$, $c\in(0,\infty)$, in $\mathbb{H}^{2n+1}$ in terms of a hyperelliptic integral.

An essential step on the proofs of Theorems \ref{t2} and \ref{t3} was to employ the Gauss--Bonnet and Chern--Gauss--Bonnet formulas for $n=1$ and $n=2$, respectively. Thus, a reasonable approach for higher dimensional cases should involve the Chern--Gauss--Bonnet formulas for locally conformally flat manifolds (see, for instance, \cite{CC,CS,V}).

Let $M_c$, $c\in(0,\infty)$, be a catenoid in $\mathbb{H}^{2n+1}$ and denote by $M_\epsilon=M_c\cap\{z\leq 1/\epsilon\}$, for $\epsilon>0$ sufficiently small, the truncated catenoid. For $M_\epsilon$, the Chern--Gauss--Bonnet formulas read 
\begin{equation}\label{CGB}
	\int_{M_\epsilon}\sigma_n(P)\,dA+\oint_{\partial M_\epsilon} S_n\,ds=\frac{(2\pi)^n}{n!}\chi(M_\epsilon)\,,
\end{equation}
where $\chi(M_\epsilon)$ is the Euler characteristic of $M_\epsilon$ (which, for $\epsilon>0$ sufficiently small, coincides with that of the catenoid $M_c$ and so is zero), $\sigma_n(P)$ is the $n$-th elementary symmetric polynomial of the eigenvalues of the Schouten tensor $P$ and $S_n$ is a boundary curvature (for details about its definition, we refer the reader to \cite{CC,CS}). 

The following result shows that the boundary integral in \eqref{CGB} does not have a constant term when considering its expansion in terms of $\epsilon>0$ and, hence, it does not contribute to the renormalized area.

\begin{lemma}\label{lemma} Let $M$ be a rotational minimal hypersurface in $\mathbb{H}^{2n+1}$ and denote by $M_\epsilon=M\cap\{z\leq 1/\epsilon\}$, $\epsilon>0$, the truncated hypersurface. Then, the expansion of 
	$$\oint_{\partial M_\epsilon} S_n\,ds\,,$$
in terms of $\epsilon>0$ has no constant term.
\end{lemma}
\textit{Proof.} Let $M_\epsilon=M\cap\{z\leq 1/\epsilon\}$ be the truncated part of a rotational minimal hypersurface in $\mathbb{H}^{2n-1}$. The expansion of the boundary integral of the statement in terms of $\epsilon>0$ is given by
\begin{equation}\label{*1}
	\oint_{\partial M_\epsilon} S_n\,ds=\sum_{i=1}^{2n-1} \frac{A_i}{\epsilon^i}+C+\mathcal{O(\epsilon)}\,,
\end{equation}where $A_i$, $i=1,...,2n-1$, and $C$ are arbitrary constants. We will show that $C=0$ must hold.

Think of $M$ as a hypersurface in $\mathbb{H}^{2n-1}$, identified by the higher dimensional analogue of \eqref{ident} as a hypersurface in the ball model, and consider its double $\widetilde{M}$. Since $M$ is a rotational minimal hypersurface, it approaches the ideal boundary at a right angle. Hence, $\widetilde{M}$ is a sufficiently regular closed hypersurface.

Denote by $M_\epsilon^+$ to $M_\epsilon$ and by $M_\epsilon^-$ the corresponding reflected part on the double of $\widetilde{M}$. It then follows that
$$\lim_{\epsilon\to 0^+} M_\epsilon^+\cup M_\epsilon^-=\widetilde{M}\,.$$
We will apply the Chern--Gauss--Bonnet formulas for locally conformally flat manifolds (see \eqref{CGB}) to $M_\epsilon^+\cup M_\epsilon^-$ and to $\widetilde{M}$ and compare them in the limit when $\epsilon\to 0^+$. For the closed hypersurface $\widetilde{M}$, we have
\begin{equation}\label{*2}
	\frac{\left(2\pi\right)^n}{n!}\chi(\widetilde{M})-\int_{\widetilde{M}} \sigma_n(P)\,dA=0\,.
\end{equation}
On the other hand, for $\widetilde{M}_\epsilon=M_\epsilon^+\cup M_\epsilon^-$, we obtain
\begin{eqnarray}\label{*3}
	2\sum_{j=1}^{n-1}\frac{A_{2j}}{\epsilon^{2j}}+2C+\mathcal{O}(\epsilon)&=&\oint_{\partial M_\epsilon^+} S_n\,ds+\oint_{\partial M_\epsilon^-} S_n\,ds=\oint_{\partial\widetilde{M}_\epsilon} S_n\,ds\nonumber\\
	&=&\frac{\left(2\pi\right)^n}{n!}\chi(\widetilde{M}_\epsilon)-\int_{\widetilde{M}_\epsilon}\sigma_n(P)\,dA\,,
\end{eqnarray}
where the first equality follows from \eqref{*1} and the fact that the expressions of the integrals over $\partial M_\epsilon^+$ and $\partial M_\epsilon^-$ are equal after the change $\epsilon\mapsto -\epsilon$. Hence, the odd powers of $\epsilon$ cancel out, while the even powers appear twice. (When $n=1$, the sum on the left of \eqref{*3} does not appear.)

Finally, observe that taking the limit when $\epsilon\to 0^+$, the right-hand side of \eqref{*3} becomes \eqref{*2} and so it vanishes. In particular, we deduce that $C=0$ proving the result. \hfill$\square$
\\

We will next compute the $n$-th elementary symmetric polynomial of the eigenvalues of the Schouten tensor $P$, namely, $\sigma_n(P)$. Recall that the Schouten tensor is defined as
\begin{equation}\label{Schouten}
	P(e_i,e_j)=\frac{1}{2(n-1)}\left({\rm Ric}(e_i,e_j)-n\,\lambda\, \delta_{ij}\right)\,,
\end{equation}
for any orthonormal frame $\{e_1,...,e_{2n}\}$, where ${\rm Ric}$ is the Ricci tensor and $\delta_{ij}$ denotes the Kronecker's delta.

The $n$-th elementary symmetric polynomial $\sigma_n$ can be obtained recursively through the Newton--Girard's identities. Denote by $p_l$, $l=1,...,n$, the power sum symmetric polynomial of degree $l$ in $2n$ variables $\mu_1,...,\mu_{2n}$. That is,
\begin{equation}\label{powersum}
	p_l(\mu_1,...,\mu_{2n})=\sum_{i=1}^{2n}\mu_i^l\,.
\end{equation}
The Newton--Girard's identities are then given by
\begin{equation}\label{NG}
	l\sigma_l=\sum_{j=1}^l(-1)^{j-1}\sigma_{l-j}p_j\,,
\end{equation}
for all $l=1,...,n$. Since $\sigma_0\equiv 1$, the identities \eqref{NG} completely describe recursively the elementary symmetric polynomials up to order $n$ in terms of the power sum polynomials.

In the following result we obtain a general formula for the power sum polynomials of the Schouten tensor $P$.

\begin{lemma}\label{pP} Let $M$ be a rotational minimal hypersurface in $\mathbb{H}^{2n+1}$ and denote by $P$ the Schouten tensor. The power sum symmetric polynomials of $P$ are given by
\begin{equation*}
	p_l(P)=\left(\frac{-1}{2}\right)^l\,\sum_{i=0}^l \begin{pmatrix} l \\ i \end{pmatrix}\frac{(4n-1)^i+(-1)^i(2n-1)}{(2n-1)^{2i}}\,\kappa^{2i}\,,
\end{equation*}
for every $l=1,...,n$. In the above expression $\kappa$ represents the curvature of the generating curve of the hypersurface $M$.
\end{lemma}
\textit{Proof.} To prove the result, observe first that the power sum symmetric polynomial $p_l(P)$ is just the trace of $P^l$ and so it is independent of the choice of orthonormal frame. Hence, for simplicity, we will consider here the orthonormal frame $\{e_1,...,e_{2n}\}$ composed by the principal directions of the rotational minimal hypersurface $M$ in $\mathbb{H}^{2n+1}$.

Employing the Gauss' equation, we can rewrite the Ricci tensor ${\rm Ric}$ as
\begin{eqnarray*}
	{\rm Ric}(e_i,e_j)&=&\sum_{l=1}^{2n}{\rm Rm}(e_i,e_l,e_j,e_l)\\
	&=&\sum_{l=1}^{2n}\left(\widetilde{\rm Rm}(e_i,e_l,e_j,e_l)+B(e_i,e_j)B(e_l,e_l)-B(e_l,e_j)B(e_i,e_l)\right)\\
	&=&-(2n-1)\delta_{ij}-\left(B^2\right)(e_i,e_j)=-(2n-1+\kappa_i^2)\delta_{ij}\,.
\end{eqnarray*}
The third equality follows from the fact that the sectional curvature of the ambient space $\mathbb{H}^{2n+1}$ is $-1$, from the minimality of the hypersurface $M$ (and so ${\rm Trace}(B)=2nH=0$), and from the definition of $B^2$; while the last equality is a consequence of $\{e_1,...,e_{2n}\}$ being the frame of principal directions and, hence, $\left(B^2\right)(e_i,e_j)=\kappa_i^2\delta_{ij}$ where $\kappa_i$ are the principal curvatures. 

Substituting this in \eqref{Schouten} and employing the relation \eqref{lambdaB} to rewrite $\lambda$ in terms of $\lvert B\rvert^2$, we get
\begin{eqnarray*}
	P(e_i,e_j)&=&-\frac{2n-1+\kappa_i^2}{2(n-1)}\,\delta_{ij}-\frac{n}{2(n-1)}\left(-1-\frac{1}{2n(2n-1)}\lvert B\rvert^2\right)\delta_{ij}\\
	&=&\left(-\frac{1}{2}+\frac{1}{2(n-1)}\left(\kappa_i^2+\frac{n}{(2n-1)^2}\kappa^2\right)\right)\delta_{ij}\,,
\end{eqnarray*}
since $\lvert B\rvert^2=2n\kappa^2/(2n-1)$ as shown in \eqref{B2}. Recall that for a rotational minimal hypersurface, we have $\kappa_1=\kappa$ and $\kappa_2=\cdots=\kappa_{2n}=-\kappa/(2n-1)$ (see \eqref{pk}). Thus, we distinguish between these cases, obtaining
$$P_{11}=P(e_1,e_1)=-\frac{1}{2}\left(1+\frac{4n-1}{(2n-1)^2}\,\kappa^2\right),$$
and
$$P_{ii}=P(e_i,e_i)=-\frac{1}{2}\left(1-\frac{1}{(2n-1)^2}\,\kappa^2\right),$$
for every $i=2,...,2n$.

Finally, from the definition \eqref{powersum}, we compute
\begin{eqnarray}\label{pk2}
	p_l(P)&=&\sum_{i=1}^{2n} P_{ii}^l=P_{11}^l+\sum_{i=2}^{2n}P_{ii}^l\nonumber\\
	&=&\left(\frac{-1}{2}\right)^l\left(1+\frac{4n-1}{(2n-1)^2}\kappa^2\right)^l+(2n-1)\left(\frac{-1}{2}\right)^l\left(1-\frac{1}{(2n-1)^2}\kappa^2\right)^l\\
	&=&\left(\frac{-1}{2}\right)^l \sum_{j=0}^l \begin{pmatrix} l \\ j \end{pmatrix}\frac{(4n-1)^j \kappa^{2j}}{(2n-1)^{2j}}+(2n-1)\left(\frac{-1}{2}\right)^l\sum_{j=0}^l\begin{pmatrix} l \\ j\end{pmatrix}(-1)^j\frac{\kappa^{2j}}{(2n-1)^{2j}}\nonumber\\
	&=&\left(\frac{-1}{2}\right)^l\sum_{j=0}^l \begin{pmatrix} l \\ j \end{pmatrix}\frac{(4n-1)^j+(-1)^j(2n-1)}{(2n-1)^{2j}}\,\kappa^{2j}\,,\nonumber
\end{eqnarray}
proving the statement. \hfill$\square$
\\

Using the expression of the power sum polynomials $p_l(P)$ of the Schouten tensor obtained in Lemma \ref{pP}, and recursively applying the Newton--Girard's identities \eqref{NG}, we can express the symmetric polynomials $\sigma_l(P)$ of the Schouten tensor of a rotational minimal hypersurface $M$ in $\mathbb{H}^{2n+1}$ in terms of the curvature $\kappa$ of its generating curve $\gamma$.

\begin{lemma}\label{sigmaP} Let $M$ be a rotational minimal hypersurface in $\mathbb{H}^{2n+1}$ and denote by $P$ the Schouten tensor. The $l$-th elementary symmetric polynomial of $P$ is given by
	\begin{equation*}
		\sigma_l(P)=\frac{(-1)^l(2n)!}{2^l\, l! (2n-l)!}\left(1+\frac{(2l-1)\kappa^2}{(2n-1)^2}\right)\left(1-\frac{\kappa^2}{(2n-1)^2}\right)^{l-1}\,,
	\end{equation*}
	for every $l=1,...,n$. In particular,
	\begin{equation*}
		\sigma_n(P)=\frac{(-1)^n(2n)!}{2^n (n!)^2}\left(1+\frac{\kappa^2}{2n-1}\right)\left(1-\frac{\kappa^{2}}{(2n-1)^2}\right)^{n-1}\,.
	\end{equation*}
	In the above expressions $\kappa$ represents the curvature of the generating curve of the hypersurface $M$.
\end{lemma}
\textit{Proof.} The result follows by induction over $l$. We begin checking the case $l=1$. For this specific value, $\sigma_1(P)=p_1(P)=-n-n\kappa^2/(2n-1)^2$, which coincides with the formula in the statement.

We next assume the formula holds for every positive integer smaller than $l$, or equivalently, that
\begin{equation}\label{hypothesis}
	\sigma_{l-j}(P)=\frac{(-1)^{l-j}(2n)!}{2^{l-j}(2n-l+j)!(l-j)!}\left(1+\left(2l-2j-1\right)v\right)\left(1-v\right)^{l-j-1},
\end{equation}
holds for every $j=1,...,n$. Here, we have introduced the simplification $v=\kappa^2/(2n-1)^2$.

In order to prove the case $l$ (i.e., $j=0$), we will employ the Newton--Girard's identities \eqref{NG} and the expression of the power sum polynomials $p_j(P)$ of Lemma \ref{pP}. More precisely, we will use \eqref{pk2} for $j$ instead of $l$. Then, from \eqref{NG} we have
\begin{equation}\label{comp}
	\sigma_l(P)=\frac{1}{l}\sum_{j=1}^l (-1)^{j-1}\sigma_{l-j}(P)p_j(P)=\frac{(-1)^l (2n)!}{l\,2^l}\left(1-v\right)^{l-1}\left(\mathcal{S}_1+(2n-1)\mathcal{S}_2\right),
\end{equation}
where we have used \eqref{pk2}, the hypothesis of induction \eqref{hypothesis}, and defined the sums
$$\mathcal{S}_1=\sum_{j=1}^l (-1)^{j-1}\frac{1+(2l-2j-1)v}{(2n-l+j)!(l-j)!}\left(\frac{1+(4n-1)v}{1-v}\right)^j\,,$$
and
$$\mathcal{S}_2=\sum_{j=1}^l(-1)^{j-1}\frac{1+(2l-2j-1)v}{(2n-l+j)!(l-j)!}\,.$$

Observe that shifting the index of summation to $i=l-j$, the sums $\mathcal{S}_1$ and $\mathcal{S}_2$ can be explicitly computed, obtaining, respectively
\begin{equation}\label{sum1}
	\mathcal{S}_1=(-1)^{l-1}w^l\sum_{i=0}^{l-1}(-1)^i\frac{1+(2i-1)v}{i! (2n-i)! w^i}=\frac{1+(4n-1)v}{2n (l-1)!(2n-l)!}\,,
\end{equation}
where we are denoting by $w=(1+(4n-1)v)/(1-v)$, for simplicity, and
\begin{equation}\label{sum2}
	\mathcal{S}_2=(-1)^{l-1}\sum_{i=0}^{l-1}(-1)^i\frac{1+(2i-1)v}{i! (2n-i)!}=\frac{2n-1+(1-6n+4nl)v}{2n(2n-1)(l-1)!(2n-l)!}\,.
\end{equation}
The last equality of both \eqref{sum1} and \eqref{sum2}, respectively, can easily be shown by induction over $l$.

Finally, combining \eqref{sum1} and \eqref{sum2} with \eqref{comp}, we conclude with
\begin{eqnarray*}
	\sigma_l(P)&=&\frac{(-1)^l(2n)!}{l\,2^l}(1-v)^{l-1}\left(\frac{1+(4n-1)v}{2n(l-1)!(2n-l)!}+\frac{2n-1+(1-6n+4nl)v}{2n(l-1)!(2n-l)!}\right)\\
	&=&\frac{(-1)^l(2n)!}{l\,2^l}(1-v)^{l-1}\frac{2n(1+(2l-1)v)}{2n(l-1)!(2n-l)!}\,,
\end{eqnarray*}
which proves the first statement. For the second, we just particularize the above identity for $l=n$. \hfill$\square$
\\

Combining the Chern--Gauss--Bonnet formula \eqref{CGB} for catenoids $M_c$, $c\in(0,\infty)$, in $\mathbb{H}^{2n+1}$ with the expression of $\sigma_n(P)$ given in Lemma \ref{sigmaP}, we can obtain a general formula for the renormalized area of $M_c$.

\begin{theorem}\label{t4} Let $M_c$, $c\in(0,\infty)$, be a non-totally geodesic rotational minimal hypersurface in $\mathbb{H}^{2n+1}$. The renormalized area $\mathcal{A}_R$ of $M_c$ is given by the hyperelliptic integral
$$\mathcal{A}_R(M_c)=\frac{2\pi^n}{(n-1)!\,c^{2n-1}}\int_0^\alpha \frac{q_n(u)}{\sqrt{u(-u^{2n}+c^2 u+(2n-1)^2)}}\,du\,,$$
where $\alpha$ is the only positive solution of $Q_{n,c}(u)=-u^{2n}+c^2 u+(2n-1)^2=0$, and $q_n$ is the polynomial of degree $n(2n-1)$ given by
$$q_n(u)=\sum_{i=1}^n(-1)^{i}\begin{pmatrix} n \\ i \end{pmatrix}(2n-1)^{2(n-i)}(2i-1)u^{n(2i-1)}\,.$$
\end{theorem}
\textit{Proof.} Let $M_c$, $c\in(0,\infty)$, be a catenoid in $\mathbb{H}^{2n+1}$. The $n$-th elementary symmetric polynomial of the Schouten tensor $P$ is given in Lemma \ref{sigmaP}. Expanding that expression by means of the binomial theorem, we compute
\begin{eqnarray*}
	\frac{(-2)^n(n!)^2}{(2n)!}\,\sigma_n(P)&=&\left(1+\frac{\kappa^2}{2n-1}\right)\sum_{i=0}^{n-1}\begin{pmatrix} n-1 \\ i \end{pmatrix}(-1)^i\frac{\kappa^{2i}}{(2n-1)^{2i}}\\
	&=&\sum_{i=0}^{n-1}\begin{pmatrix} n-1 \\ i \end{pmatrix}(-1)^i\frac{\kappa^{2i}}{(2n-1)^{2i}}+\sum_{j=0}^{n-1}\begin{pmatrix} n-1 \\ j \end{pmatrix}(-1)^j\frac{\kappa^{2j+2}}{(2n-1)^{2j+1}}\\
	&=&1+\sum_{i=1}^{n-1}\begin{pmatrix} n-1 \\ i \end{pmatrix}\frac{(-1)^i\kappa^{2i}}{(2n-1)^{2i}}-(2n-1)\sum_{i=1}^n\begin{pmatrix} n-1 \\ i-1 \end{pmatrix}\frac{(-1)^i\kappa^{2i}}{(2n-1)^{2i}}\\
	&=&1-(-1)^n\frac{\kappa^{2n}}{(2n-1)^{2n-1}}-\sum_{i=1}^{n-1}\begin{pmatrix} n \\ i \end{pmatrix} (-1)^i(2i-1)\frac{\kappa^{2i}}{(2n-1)^{2i}}\,.
\end{eqnarray*}
Of course, when $n=1$, it must be understood that the second term after the third equality does not appear.

For $\epsilon>0$ sufficiently small, denote by $M_\epsilon=M_c\cap \{z\leq 1/\epsilon\}$ the truncated catenoid. Integrating the above relation over $M_\epsilon$ and employing the Chern--Gauss--Bonnet formula \eqref{CGB}, we obtain (after rearranging)
$$\mathcal{A}(M_\epsilon)=-\frac{(-2)^n(n!)^2}{(2n)!}\oint_{\partial M_\epsilon} S_n\,ds+\frac{(-1)^n}{(2n-1)^{2n-1}}\int_{M_\epsilon} \kappa^{2n}dA+\sum_{i=1}^{n-1}\frac{(-1)^i(2i-1)}{(2n-1)^{2i}}\begin{pmatrix} n \\ i \end{pmatrix}\int_{M_\epsilon}\kappa^{2i}dA.$$
Hence, looking at the constant term when expanding $\mathcal{A}(M_\epsilon)$ in terms of $\epsilon>0$, we deduce that
$$\mathcal{A}_R(M_c)=\frac{(-1)^n}{(2n-1)^{2n-1}}\int_{M_c}\kappa^{2n}\,dA+\sum_{i=1}^{n-1}\frac{(-1)^i(2i-1)}{(2n-1)^{2i}}\begin{pmatrix} n \\ i \end{pmatrix} \int_{M_c}\kappa^{2i}\,dA\,,$$
since from Lemma \ref{lemma}, the boundary integral does not have a constant term. To show that the above expression is convergent we will first transform the integrals on the right-hand side into those of the statement. Then, the convergence will follow as in Remark \ref{convergence}.

Recall now that Theorem \ref{t1} shows that $M_c$ is the warped product hypersurface $\gamma_c\times_f \mathbb{S}^{2n-1}$, where the generating curve $\gamma_c$ is a $n$-catenary with curvature $\kappa=u^n$ and $f=(2n-1)(c^2u)^{-1/2}$. It then follows that $dA=(2n-1)^{2n-1}(c^2u)^{-(2n-1)/2}d\overline{A}\,ds$, where $d\overline{A}$ is the (Euclidean) area element on the round sphere $\mathbb{S}^{2n-1}$. Since the area of $\mathbb{S}^{2n-1}$ is $2\pi^n/(n-1)!$, we obtain
\begin{eqnarray*}
	\mathcal{A}_R(M_c)&=&\frac{2(-\pi)^n}{(n-1)!c^{2n-1}}\int_{\gamma_c} u^{(4n^2-2n+1)/2}ds\\
	&&+\frac{2\pi^n}{(n-1)!c^{2n-1}}\sum_{i=1}^{n-1}(-1)^i(2i-1)(2n-1)^{2(n-i)-1}\begin{pmatrix} n \\ i \end{pmatrix}\int_{\gamma_c}u^{(4in-2n+1)/2}ds\\
	&=&\frac{2\pi^n}{(n-1)!c^{2n-1}}\int_0^\alpha \frac{q_n(u)}{\sqrt{u(-u^{2n}+c^2u+(2n-1)^2)}}\,du\,,
\end{eqnarray*}
where $q_n(u)$ are the polynomials defined on the statement. For the last equality above we have employed the Euler--Lagrange equation \eqref{ODE} to make a change of variable. \hfill$\square$

\begin{rem}\label{convergence} For every $c\in(0,\infty)$ and positive integer $n$ fixed, the hyperelliptic integral in the statement of Theorem \ref{t4} is convergent. 
	
In fact, even though the integrand $q_n(u)/\sqrt{u\,Q_{n,c}(u)}$ is singular at $u=0$ and $u=\alpha$, one can see that for $\epsilon>0$ sufficiently small, there exist constants $C,\widetilde{C}\in\mathbb{R}$ such that
$$\left| \frac{q_n(u)}{\sqrt{u\,Q_{n,c}(u)}}\right| \leq C\,u^{n-1/2}\,,$$
for every $u\in(0,\epsilon)$, and
$$\left| \frac{q_n(u)}{\sqrt{u\,Q_{n,c}(u)}}\right| \leq \widetilde{C}\left(\alpha-u\right)^{-1/2}\,,$$
for every $u\in(\alpha-\epsilon,\alpha)$. Therefore, $q_n(u)/\sqrt{u\,Q_{n,c}(u)}\in\mathcal{L}^1(0,\alpha)$.
\end{rem}

In the specific cases $n=1$ and $n=2$, the expression of the renormalized area $\mathcal{A}_R(M_c)$ given in Theorem \ref{t4} reduces to \eqref{test1} and \eqref{renA4}, respectively.

Furthermore, numerically representing $\mathcal{A}_R(M_c)$ as a function of $c>0$ for fixed values of $n$, the reader can observe that the renormalized area of catenoids in $\mathbb{H}^{2n+1}$ increases from $-\infty$ (this assertion will be shown in Theorem \ref{t5} below), attains its maximum precisely when the catenoid attains its maximum height $\Psi_n$, and then decreases while approaching $2(-4\pi)^nn!/(2n)!$ (see Theorem \ref{t5}) which is twice the renormalized area of $\mathbb{H}^{2n}$ (see Appendix A for details). In particular, when $n$ is even this quantity is positive and so the renormalized area of minimal hypersurfaces does not have a sign. In the following result we prove this statement.

\begin{theorem}\label{t5} The renormalized area $\mathcal{A}_R$ of non-totally geodesic rotational minimal hypersurfaces $M_c$, $c\in(0,\infty)$, in $\mathbb{H}^{2n+1}$ varies continuously from $-\infty$, when $c\to 0^+$, to
$$\lim_{c\to \infty}\mathcal{A}_R(M_c)=2\mathcal{A}_R(\mathbb{H}^{2n})=2\frac{(-4\pi)^nn!}{(2n)!}\,.$$
	
Consequently, the renormalized area of hypersurfaces in $\mathbb{H}^{2n+1}$ is not bounded below. Moreover, when $n$ is even, the renormalized area of minimal hypersurfaces does not have a sign.
\end{theorem}
\textit{Proof.} Consider the family of catenoids $M_c$, $c>0$, in $\mathbb{H}^{2n+1}$. We will see that $\mathcal{A}_R(M_c)\to-\infty$ as $c\to 0^+$ and, hence, the renormalized area of hypersurfaces is not bounded below. 

To compute the limit when $c\to 0^+$ of the expression of $\mathcal{A}_R(M_c)$ given in Theorem \ref{t4}, we will introduce the integrals
$$J_m(c)=\int_0^\alpha \frac{u^m}{\sqrt{u\,Q_{n,c}(u)}}\,du\,,$$
where $\alpha$ is the only positive root of the polynomial $Q_{n,c}(u)=-u^{2n}+c^2u+(2n-1)^2$, and $m$ is any positive integer. Then, we compute
\begin{eqnarray*}\lim_{c\to 0^+} J_m(c)&=&\int_0^{(2n-1)^{1/n}}\frac{u^m}{\sqrt{u(-u^{2n}+(2n-1)^2)}}\,du=\frac{(2n-1)^{\frac{2m-2n+1}{2n}}}{2n}\,\beta\left(\frac{2m+1}{4n},\frac{1}{2}\right)\\
&=&\frac{(2n-1)^{\frac{2m-2n+1}{2n}}\sqrt{\pi}\,\Gamma\left(\frac{2m+1}{4n}\right)}{2n\,\Gamma\left(\frac{2m+2n+1}{4n}\right)}\,,
\end{eqnarray*}
where the second equality follows after the change of variable $u=(2n-1)^{1/n}x^{1/(2n)}$, and the last equality is just the standard relation between the Beta $\beta(\cdot,\cdot)$ and Gamma $\Gamma(\cdot)$ functions together with $\Gamma(1/2)=\sqrt{\pi}$.

It then follows, applying the above limit of $J_m(c)$ for $m=n(2i-1)$ and the definition of the polynomials $q_n(u)$ given in the statement of Theorem \ref{t4}, that
\begin{eqnarray}\label{hyp}
	\lim_{c\to 0^+}\int_0^\alpha \frac{q_n(u)}{\sqrt{u\,Q_{n,c}(u)}}\,du&=&\sum_{i=1}^n(-1)^i\begin{pmatrix} n \\ i \end{pmatrix}(2n-1)^{2(n-i)}(2i-1)\lim_{c\to 0^+}J_{n(2i-1)}(c)\nonumber\\
	&=&\frac{(2n-1)^{\frac{(2n-1)^2}{2n}}\sqrt{\pi}}{2n}\sum_{i=1}^n(-1)^i\begin{pmatrix} n \\ i\end{pmatrix}(2i-1)\frac{\Gamma\left(i-\frac{1}{2}+\frac{1}{4n}\right)}{\Gamma\left(i+\frac{1}{4n}\right)}\nonumber\\
	&=&-2(2n-1)^{2n-3+\frac{1}{2n}}\,\sqrt{\pi}\,\frac{\Gamma\left(\frac{2n+1}{4n}\right)}{\Gamma\left(\frac{1}{4n}\right)}<0\,.
\end{eqnarray}
The last equality follows since the sum in the second line can be decomposed as a combination of two terminating hypergeometric functions $_2F_1(-n,\cdot,\cdot,1)$ which can be computed using Gauss' hypergeometric theorem. More precisely, rewriting
$$2i-1=2\left(i-\frac{1}{2}+\frac{1}{4n}\right)-\frac{1}{2n}\,,$$
and employing that $\Gamma(x+1)=x\,\Gamma(x)$, we have that the sum in the second line above equals
$$2\sum_{i=1}^n (-1)^i\begin{pmatrix} n \\ i\end{pmatrix}\frac{\Gamma\left(i+\frac{1}{2}+\frac{1}{4n}\right)}{\Gamma\left(i+\frac{1}{4n}\right)}-\frac{1}{2n}\sum_{i=1}^n(-1)^i\begin{pmatrix} n\\ i\end{pmatrix}\frac{\Gamma\left(i-\frac{1}{2}+\frac{1}{4n}\right)}{\Gamma\left(i+\frac{1}{4n}\right)}\,.$$
Each of these sums can in turn be rewritten as terminating hypergeometric functions, obtaining
$$\frac{2\Gamma\left(\frac{1}{2}+\frac{1}{4n}\right)}{\Gamma\left(\frac{1}{4n}\right)}\left(_2F_1\left(-n,\frac{1}{4n}+\frac{1}{2},\frac{1}{4n},1\right)-1\right)-\frac{\Gamma\left(-\frac{1}{2}+\frac{1}{4n}\right)}{2n\,\Gamma\left(\frac{1}{4n}\right)}\left(_2F_1\left(-n,\frac{1}{4n}-\frac{1}{2},\frac{1}{4n},1\right)-1\right).$$
Using Gauss' summation to compute the two hypergeometric functions above, we see that these terms cancel out. In fact,
\begin{eqnarray*}
	2\Gamma\left(\frac{1}{2}+\frac{1}{4n}\right)\,_2F_1\left(-n,\frac{1}{4n}+\frac{1}{2},\frac{1}{4n},1\right)-\frac{1}{2n}\Gamma\left(-\frac{1}{2}+\frac{1}{4n}\right)\,_2F_1\left(-n,\frac{1}{4n}-\frac{1}{2},\frac{1}{4n},1\right)&=&\\
	=2\Gamma\left(\frac{1}{2}+\frac{1}{4n}\right)\frac{\Gamma\left(\frac{1}{4n}\right)\Gamma\left(n-\frac{1}{2}\right)}{\Gamma\left(n+\frac{1}{4n}\right)\Gamma\left(-\frac{1}{2}\right)}-\frac{1}{2n}\Gamma\left(-\frac{1}{2}+\frac{1}{4n}\right)\frac{\Gamma\left(\frac{1}{4n}\right)\Gamma\left(n+\frac{1}{2}\right)}{\Gamma\left(n+\frac{1}{4n}\right)\Gamma\left(\frac{1}{2}\right)}&=&\\
	=\frac{\Gamma\left(\frac{1}{4n}\right)\Gamma\left(n-\frac{1}{2}\right)\Gamma\left(-\frac{1}{2}+\frac{1}{4n}\right)}{\sqrt{\pi}\,\Gamma\left(n+\frac{1}{4n}\right)}\left(-\left(\frac{1}{4n}-\frac{1}{2}\right)-\frac{1}{2n}\left(n-\frac{1}{2}\right)\right)&=&0\,,
\end{eqnarray*}
since $\Gamma(x+1)=x\,\Gamma(x)$, $\Gamma(1/2)=\sqrt{\pi}$, and $\Gamma(-1/2)=-2\sqrt{\pi}$.

Consequently,
\begin{eqnarray*}
	\sum_{i=1}^n (-1)^i\begin{pmatrix} n \\ i \end{pmatrix}(2i-1)\frac{\Gamma\left(i-\frac{1}{2}+\frac{1}{4n}\right)}{\Gamma\left(i+\frac{1}{4n}\right)}&=&\frac{\Gamma\left(-\frac{1}{2}+\frac{1}{4n}\right)}{2n\,\Gamma\left(\frac{1}{4n}\right)}-2\frac{\Gamma\left(\frac{1}{2}+\frac{1}{4n}\right)}{\Gamma\left(\frac{1}{4n}\right)}\\
	&=&\frac{\Gamma\left(\frac{1}{2}+\frac{1}{4n}\right)}{\Gamma\left(\frac{1}{4n}\right)}\left(\frac{1}{2n\left(-\frac{1}{2}+\frac{1}{4n}\right)}-2\right)\\
	&=&\frac{-4n}{2n-1}\frac{\Gamma\left(\frac{1}{2}+\frac{1}{4n}\right)}{\Gamma\left(\frac{1}{4n}\right)}\,,
\end{eqnarray*}
and so \eqref{hyp} holds.

Since $\mathcal{A}_R(M_c)$ has a coefficient $1/c^{2n-1}$ in front of the integral, it tends to $-\infty$ when $c\to 0^+$. This shows that the renormalized area is not bounded below and that there exist (minimal) hypersurfaces for which it is negative.

For the second assertion, we will compute the limit when $c\to\infty$ of $\mathcal{A}_R(M_c)$. Observe first that the positive solution $\alpha$ of $Q_{n,c}(u)=0$, is asymptotically $\alpha\simeq c^{2/(2n-1)}$. We can then apply the dominated convergence theorem and the change of variable $u=c^{2/(2n-1)}x$ to obtain
$$\lim_{c\to\infty}\mathcal{A}_R(M_c)=\lim_{c\to\infty} \frac{2\pi^n}{(n-1)!c^{2n}}\int_0^1 \frac{q_n(c^{2/(2n-1)}x)}{x\sqrt{1-x^{2n-1}}}\,dx=\frac{2(-\pi)^n}{(n-1)!}\int_0^1 \frac{x^{2n^2-n-1}}{\sqrt{1-x^{2n-1}}}\,dx\,,$$
since only the integral with the term of $u^{n(2n-1)}$ will not cancel out when $c\to\infty$. We next make the change of variable $y=x^{2n-1}$, to get
$$\lim_{c\to\infty}\mathcal{A}_R(M_c)=\frac{2(-\pi)^n}{(n-1)!}\int_0^1\frac{y^{n-1}}{\sqrt{1-y}}\,dy=\frac{2(-\pi)^n}{(n-1)!}\,\beta(n,1/2)=\frac{2(-\pi)^n\sqrt{\pi}}{\Gamma\left(n+1/2\right)}\,.$$
A simple argument by induction can be used to show that
$$\Gamma\left(n+\frac{1}{2}\right)=\frac{(2n)!\sqrt{\pi}}{4^n n!}\,,$$
which combined with the above gives
$$\lim_{c\to\infty}\mathcal{A}_R(M_c)=2\frac{(-4\pi)^nn!}{(2n)!}\,.$$
If $n$ is even, the above quantity is positive and, hence, there exist catenoids $M_c$ (with $c>0$ sufficiently large) in $\mathbb{H}^{2n+1}$, with $n$ even, for which the renormalized area is positive. This in combination with the limit $c\to 0^+$ proves the second statement. \hfill$\square$

\section*{Appendix A. Renormalized Area of Hyperbolic Spaces}

The purpose of this appendix is to adapt our computations regarding catenoids to obtain a formula for the renormalized area of even-dimensional hyperbolic spaces $\mathbb{H}^{2n}$. This renormalized area has been previously studied in \cite{GT,CQY,E}.

Let $\mathbb{H}^{2n}$ be a totally geodesic hypersurface of $\mathbb{H}^{2n+1}$. Most of the computations of Section 6 still valid in this context. More precisely, Lemmas \ref{lemma}, \ref{pP}, and \ref{sigmaP} hold for a totally geodesic rotational (minimal) hypersurface in $\mathbb{H}^{2n+1}$. The hypersurface $\mathbb{H}^{2n}$ can be considered as a rotational minimal hypersurface of $\mathbb{H}^{2n+1}$ where the curvature $\kappa$ of the generating curve is identically zero (in other words, the generating curve is a geodesic). From this observation and Lemma \ref{sigmaP}, we obtain that the $n$-th elementary symmetric polynomial of the Schouten tensor is
$$\sigma_n(P)=\frac{(-1)^n(2n)!}{2^n(n!)^2}\,.$$
Therefore, integrating over the truncated hypersurface $\mathbb{H}_\epsilon^{2n}=\mathbb{H}^{2n}\cap\{z\leq 1/\epsilon\}$, we get
$$\mathcal{A}(\mathbb{H}^{2n}_\epsilon)=\frac{(-2)^n(n!)^2}{(2n)!}\int_{\mathbb{H}^{2n}_\epsilon} \sigma_n(P)\,dA=\frac{(-2)^n(n!)^2}{(2n)!}\left(\frac{(2\pi)^n}{n!}-\oint_{\partial\mathbb{H}_\epsilon^{2n}} S_n\,ds\right),$$
where we have applied the Chern--Gauss--Bonnet formula \eqref{CGB}, with $\chi(\mathbb{H}_\epsilon^{2n})=1$.

From Lemma \ref{lemma}, the boundary integral above has no constant term when expanding $\mathcal{A}(\mathbb{H}_\epsilon^{2n})$ in terms of $\epsilon>0$ and so, we conclude that the renormalized area of the even-dimensional hyperbolic spaces $\mathbb{H}^{2n}$ is given by
$$\mathcal{A}_R(\mathbb{H}^{2n})=\frac{(-4\pi)^n n!}{(2n)!}\,.$$
Compare with Corollary 3.5 of \cite{CQY} and \cite{E}.

\bigskip

\begin{flushleft}
	\'Alvaro P{\footnotesize \'AMPANO}\\
	Department of Mathematics and Statistics, Texas Tech University, Lubbock, TX, 79409, USA\\
	E-mail: alvaro.pampano@ttu.edu
\end{flushleft}

\begin{flushleft}
	Aaron J. T{\footnotesize YRRELL}\\
	Department of Mathematics, University of Notre Dame, Notre Dame, IN 46556, USA\\
	E-mail: atyrrell@nd.edu
\end{flushleft}

\end{document}